\documentclass[a4paper]{article}
\usepackage[a4paper, margin=2cm]{geometry}

\usepackage[utf8]{inputenc}

\usepackage{mathtools}
\usepackage{amssymb}
\usepackage{amsthm}

\usepackage[inline]{enumitem}
\usepackage{tikz-cd}

\usepackage[hidelinks]{hyperref}
\usepackage{cleveref}

\usepackage{xcolor}
\hypersetup{
    colorlinks,
    linkcolor={red!50!black},
    citecolor={blue!50!black},
    urlcolor={blue!80!black}
}

\newcommand{\ovl}[1]{{\overline{#1}}}
\newcommand{\ovs}[2]{{\overset{#1}{#2}}}

\newcommand{\NN}{\mathbb{N}}

\newcommand{\ZZ}{\mathbb{Z}}

\newcommand{\RR}{\mathbb{R}}

\DeclareMathOperator{\spn}{span}

\DeclareMathOperator{\dom}{dom}

\DeclareMathOperator{\prj}{prj}

\DeclareMathOperator{\Id}{Id}

\DeclareMathOperator{\sgn}{sgn}
\newcommand{\facemorphism}[1]{s(#1)}
\newcommand{\morphismface}[1]{s^{-1}(#1)}
\newcommand{\st}[1]{\llcorner {#1}\urcorner }

\newcommand{\REZ}{\nabla ^{\ovl{N}}}
\DeclareMathOperator{\Mor}{Mor}
\newcommand{\Met}{\mathsf{Met}}
\newcommand{\Bet}{\mathsf{Bet}}
\newcommand{\ChZ}{\mathsf{Ch}_\ZZ }
\newcommand{\ptsSet}{\mathsf{sSet}_*}

\newtheorem{theorem}{Theorem}[section]
\newtheorem{corollary}[theorem]{Corollary}
\newtheorem{lemma}[theorem]{Lemma}
\newtheorem{proposition}[theorem]{Proposition}
\newtheorem{example}[theorem]{Example}
\newtheorem{remark}[theorem]{Remark}
\theoremstyle{definition}
\newtheorem{definition}[theorem]{Definition}

\newcommand{\MH}{\mathrm{MH}}

\newcommand{\MC}{\mathrm{MC}}
\newcommand{\MZ}{\mathrm{MZ}}
\newcommand{\MS}{\mathrm{MS}}
\newcommand{\MB}{\mathrm{MB}}
\newcommand{\NR}{\mathrm{\ovl{N}}}
\newcommand{\pp}{\mathbf{p}}
\newcommand{\xx}{\mathbf{x}}
\newcommand{\yy}{\mathbf{y}}
\newcommand{\zz}{\mathbf{z}}
\newcommand{\ppath}{path}
\newcommand{\Paths}{P}
\newcommand{\pt}{\mathrm{pt}}
\newcommand{\ddd}{\mathsf{d}}
\newcommand{\sss}{\mathsf{s}}
\newcommand{\hG}{\hat{G}}

\begin{document}

    \title{
        Magnitude Homology, Diagonality, Medianness, Künneth and Mayer-Vietoris\\
    }
    \author{Rémi Bottinelli \& Tom Kaiser \footnote{\texttt{first.last@unine.ch}}}
    \maketitle

    \paragraph{Abstract.}
    Magnitude homology of graphs was introduced by Hepworth and Willerton in~\cite{HW}.
    Magnitude homology of arbitrary metric spaces by Leinster and Shulman in~\cite{LS}.
    We verify that the Künneth and Mayer-Vietoris formulas proved in~\cite{HW} for graphs adapt naturally to the metric setting.
    The same is done for the notion of diagonality, also originating from~\cite{HW}.
    Stability of this notion under products, retracts, filtrations is verified, and, as an application, it is shown that median spaces are diagonal; in particular any Menger convex median space has vanishing magnitude homology.
    Finally, we argue for a definition of magnitude homology in the context of “betweenness spaces” and develop some of its properties.

    \paragraph{Acknowledgements.} Victor Chepoi for making the content of \Cref{proposition:finite_median_spaces_are_almost_graphs} known to us.


    \section{Introduction}

    These notes have three somewhat distinct purposes.
    
    For the first, after remarking that median graphs have diagonal magnitude homology, one is led to wonder if the same could be said of median metric spaces.
    The first obvious obstacle to an answer is the fact that diagonality is only defined for graphs, but after reinterpreting and recasting it in the metric language, the question can be asked again.
    By verifying stability properties for diagonality and using an equivalence between finite median metric spaces and finite median graphs, one is finally led to a positive answer.

    The second is the verification that the Künneth and Mayer-Vietoris formulae for graph magnitude homology do indeed work in the metric setting (with appropriate adaptations).
    This is essentially doing the grunt work of going through the proofs of~\cite{HW} and checking that everything said of graphs still makes sense for metric spaces, and adapting arguments when needed. 

    Finally, while writing these notes, we remarked that once \emph{betweenness}\footnote{$y$ is between $x$ and $z$ if $d(x,y)+d(y,z)=d(x,z)$.} is defined, most arguments can be worked out without appealing to either notion of length or distance, instead relying only on betweenness.
    Thus, we strove to make this reliance on betweenness as apparent as possible, while de-emphasising the length grading.

    The last section is dedicated to developing this idea:
    We introduce two tentative “betweenness” categories for which the magnitude complex naturally applies, and relate them to the usual magnitude complex on metric spaces. 

    We now briefly introduce our working definitions; more motivation is to be found in~\cite{HW} and~\cite{LS}.

    \subsection{Metric Spaces}

    We consider only classical\footnote{The distance function is non-degenerate, $\RR_{\geq 0}$-valued symmetric and satisfies the triangle inequality.} metric spaces for simplicity's sake, although we expect everything said below to hold when infinite distances are allowed.
    Sequences of points in $X$ are written using pointy brackets $\xx = \langle x_0,\dots ,x_k\rangle $. 
    If consecutive elements are different ($\forall 0\leq i<k:\ x_i\neq x_{x+1}$), we call such a sequence a \emph{($k$-)\ppath{}}; the set of $k$-\ppath{}s in $X$ is written $\Paths_k(X)$, and the set of all \ppath{}s $\Paths(X)$.
    Given two points $x,y\in X$, we say that a third point $z\in X$ lies \emph{between} them if
    \[
        d(x,y) = d(x,z) + d(z,y).
    \]
    In other words, $z$ realises the triangle inequality.
    If furthermore $z\neq x,y$, we say that $z$ lies \emph{strictly between} $x$ and $y$.
    We write $[x,y]$ for the points between $x$ and $y$ and $]x,y[$ those strictly between.
    We call $[x,y]$ and $]x,y[$ \emph{intervals} for obvious reasons.
    A $k$-\ppath{} $\xx$ is \emph{saturated} if each strict interval $]x_i,x_{i+1}[$ is empty.
    A metric space is \emph{Menger convex} if no strict interval is empty.

    A map $f:X\rightarrow Y$ between metric spaces is non-expanding (or $1$-Lipschitz) if for all $x,x'\in X$, we have
    \[
        d(fx,fx') \leq  d(x,x').
    \]

    A subset $A$ of a metric space $X$ is convex if for all $a,b\in A$, the interval $[a,b]$ in $X$ is contained in $A$; in other words, any point between points of $A$ is in $A$.
    Note that this definition is stronger than the one found in~\cite{HW} for graphs.

    If $X$ is a set and $A$ is a net, we call a filtration $(U_\alpha )_{\alpha \in A}$ a sequence of subsets of $X$ with $\bigcup _\alpha  U_\alpha  = X$ and such that for any $\alpha \leq \beta $, we have $U_\alpha  \subseteq  U_\beta $. 
    
    If $Y$ is a subspace of $X$, a retraction of $X$ onto $Y$ is a $1$-Lipschitz map $f:X\rightarrow Y$ satisfying $f|Y = \Id_Y$.

    \subsection{Graphs}

    We view graphs as metric spaces with distance valued in $\NN$ and induced by edges.
    A $1$-Lipschitz map is therefore a map from vertices to vertices that either leaves edges intact or collapses them.
    Since we do not work with infinite distances, we assume our graphs are connected.

    When considering a subset of vertices of a graph, the distance on this subset is the restriction of the original distance to the subset, not the induced distance.
    In particular, subsets of vertices are not interpreted as subgraphs.

    Since our working definitions differ from the ones in~\cite{HW}, some care has to be taken.

    \subsection{Magnitude Homology}

    \begin{definition}[Magnitude complex]
        Let $X$ be a metric space.
        The \emph{magnitude complex} of $X$ is the chain complex whose $k$-th module is:
        \[
            \MC_k(X) := \ZZ \{\xx\ :\ \text{a $k$-\ppath{} in $X$}\}
        \]
        (the free abelian group on the set of all $k$-\ppath{}s in $X$) and with boundary map
        \[
            \partial _k := \sum _{i=1}^{k-1} (-1)^i\partial _{k,i}: \MC_k(X) \rightarrow  \MC_{k-1}(X),
        \]
        where $\partial _{k,i}$ is defined as:
        \[
            \partial _{k,i} \langle x_0,\dots ,x_i,\dots ,x_k\rangle  := 
            \begin{cases}
                \langle x_0,\dots ,\widehat{x_i},\dots ,x_k\rangle   & \text{ if $x_i$ is between $x_{i-1}$ and $x_{x+1}$;}\\
                0                       & \text{ otherwise.}
            \end{cases}
        \]
        We will write $\MZ_k(X)$ for the kernels and $\MB_k(X)$ for the images of $\partial $ respectively, so that magnitude homology becomes
        \[
            \MH_k(X) := \MZ_k(X)/\MB_k(X).
        \]
        The magnitude complex enjoys a grading on $\RR ^{\geq 0}$ by letting $\MC_k^l(X)$ be spanned by the $k$-\ppath{}s of length $l$.

        A non-expanding ($1$-Lipschitz) map $f:X\rightarrow Y$ induces a morphism of magnitude complexes by letting:
        \[
            \MC_k(f)\langle x_0,\dots ,x_k\rangle  := 
            \begin{cases}
                \langle fx_0,\dots ,fx_k\rangle     & \text{ if $l\langle fx_0,\dots ,fx_k\rangle  = l\langle x_0,\dots ,x_k\rangle $;}\\
                0               & \text{ otherwise.}
            \end{cases}
        \]
        
    \end{definition}
    Thus, $\MC_*(\bullet )$ defines a functor from the category of metric spaces with non-expanding maps to the category of $\ZZ $-modules with $\RR ^{\geq 0}$-grading.

    \begin{remark}[On length]
        In the remainder of these notes, we will try to de-emphasise the “length grading” of magnitude homology.
        On the one hand, the graded versions can in all cases at hand be recovered easily.
        On the other, the homogenisation obtained by forgetting about length allows both leaner formulations of the main results (Künneth, Excision, Diagonality) and gives us more flexibility.
        In particular, when proving diagonality of median spaces, we make use of the equivalence between finite median spaces and finite median graph in terms of betweenness; this would not work while trying to preserve length.
        
        In short, it seems to us that the notion of \emph{betweenness} plays a more important role than length itself in magnitude homology.
        
        The part where length \emph{is} important is in the definition of induced maps in magnitude for $1$-Lipschitz maps. 
        See the last section for an attempt at a length-free approach.
    \end{remark}

    \section{Median Graphs}

    \begin{definition}[Median graphs]
        A graph $X$ is said to be \emph{median} if, for any three pairwise distinct points $x,y,z$, the intersection $[x,y]\cap [y,z]\cap [x,z]$ consists of a single point, written $m(x,y,z)$.
    \end{definition}
        
    In this section, and as preparation for the next ones, we prove the following:
    \begin{proposition}\label{proposition:median_graphs_are_diagonal}
        Median graphs are diagonal (in the sense of Hepworth and Willerton).
    \end{proposition}

    The following beautiful characterisation of median graphs, due to Bandelt, will be crucial in the sequel:
    \begin{theorem}[{\cite[Theorem 2]{B}}]
        Median graphs are precisely the retracts of hypercubes.
    \end{theorem}
    
    Note that in the above, no restriction on cardinality is imposed.
    
    We now only need three simple properties of diagonality, whose proofs we will not linger on, since generalisations will come in~\Cref{section:diagonality}.
    
    \begin{proposition}[{\cite[Proposition 35]{HW}}]\label{proposition:products_of_diagonal_graphs_are_diagonal}
        Cartesian products of diagonal graphs are diagonal.
    \end{proposition}
    \begin{proof}
        By applying the Künneth formula and noting that diagonal graphs have torsion-free homologies.
        See~\cite[Proposition 35]{HW}.
    \end{proof}
    \begin{proposition}\label{proposition:retracts_of_diagonal_graphs_are_diagonal}
        Retracts of diagonal graphs are diagonal.
    \end{proposition}
    \begin{proof}
        If $f:X\rightarrow Y$ is a retraction, it has left inverse the inclusion $\iota :Y\rightarrow X$, by definition.
        Functoriality of $\MH_k^l$ implies that $\MH_k^l(f):\MH_k^l(X)\rightarrow \MH_k^l(Y)$ is surjective.
        In particular, $\MH_k^l$ being zero outside the diagonal for $X$ implies the same for $Y$.
    \end{proof}
    \begin{proposition}\label{proposition:filtrations_of_diagonal_graphs_are_diagonal}
        Graphs with filtrations by diagonal graphs are diagonal. 
    \end{proposition}
    \begin{proof}
        Let $(U_\alpha )_\alpha $ a filtration of $X$.
        Then $\MZ_k^l(X) = \bigcup _\alpha  \MZ_k^l(U_\alpha )$ and $\MB_k^l(X) = \bigcup _\alpha  \MB_k^l(U_\alpha )$.
        If for all $\alpha $, and $k\neq l$, $\MZ_k^l(U_\alpha ) = \MB_k^l(U_\alpha )$, then $\MZ_k^l(X) = \MB_k^l(X)$ and the homology vanishes outside the diagonal.
    \end{proof}

    We can now proceed with the proof:
    \begin{proof}[Proof of \Cref{proposition:median_graphs_are_diagonal}]
        Fix $X$ a median graph and $Q$ a hypercube of which $X$ is a retract.
        $Q$ has a filtration by finite hypercubes, which are diagonal (\Cref{proposition:products_of_diagonal_graphs_are_diagonal}); hence so is $Q$ (\Cref{proposition:filtrations_of_diagonal_graphs_are_diagonal}), and $X$ (\Cref{proposition:retracts_of_diagonal_graphs_are_diagonal}).
    \end{proof}

    \section{Künneth and Excision}
    
    In~\cite{HW}, Hepworth and Willerton described versions of the Künneth and Excision theorems applying to magnitude homology of graphs.
    In~\cite{LS}, Leinster and Shulman, extending magnitude homology to metric spaces\footnote{among other things}, asked whether Künneth and Excision extend to this new setting.

    The answer is yes, assuming the right reinterpretations.

    \begin{proposition}[Künneth theorem – metric setting]
        Hepworth and Willerton's statement and proof of the Künneth theorem (\cite[Theorem 21]{HW}) extends verbatim to $l^{1}$ products of metric spaces. 
    \end{proposition}

    Beware that in the category of metric spaces and non-expanding maps, the $l^{1}$ product is \emph{not} the categorical product.

    \begin{proposition}[Excision formula – metric case]
        Hepworth and Willerton's statement and proof of the Excision formula (\cite[Theorem 29]{HW}) extends to so-called \emph{gated} decompositions of metric spaces with minimal changes. 
        The same is true for the Mayer-Vietoris theorem.
    \end{proposition}
    
    The “minimal changes” understood above are a bit trickier than simple generalisations.
    In particular, the “Metric Excision formula” that we define is not strictly a generalisation of the graph theoretic one of~\cite{HW}, since the definitions we work with are not themselves generalisations of the ones in~\cite{HW}.

    A \emph{gated} decomposition in metric spaces is essentially the translation of a “projecting decomposition” into the metric language: 
    A triple $(X;Y,Z)$ with $X=Y\cup Z$ such that for any point of $Z$, there exists a unique so-called gate for this point lying between it and any point of $Y\cap Z$.
    See~\cite{DS} for details on gates.

    Since our arguments mainly consist in tweaking the original constructions of Hepworth and Willerton, having a copy of~\cite{HW} at hand will prove useful!

    \subsection{Künneth}

    If $X,Y$ are metric spaces, we recall that the Cartesian product $X\times Y$ is endowed with the $l^{1}$ metric:
    \[
        d((x,y),(x',y')) := d(x,x') + d(y,y').
    \]
    This implies that the intervals satisfy the identity:
    \[
        [(x,y),(x',y')] = [x,y]\times [x',y'],
    \]
    which is the reason for the $l^{1}$ metric's appearance in this context.
    Note also that the $l^{1}$ product reduces to the usual Cartesian product in the case of graphs.

    \paragraph{Summary of differences.}
    The arguments in~\cite[Section 8]{HW} go through verbatim for proving the Künneth formula in the case of metric spaces, since the main ingredient is the notion of betweenness, which generalises directly from graphs to metric spaces. 
    Our downplaying of length as a grading of magnitude homology simplifies some expressions by virtue of getting rid of some $\bigoplus _l$s and $\bigvee _l$s; this is syntactical.
    Other than that, we chose to put more emphasis on some arguments that could be suspected of hiding complications; conversely others that are clearly independent of the metric/graph schism are only glossed over.
    In short, our arguments do not provide any new insight, but merely confirm that the generalisation holds.

    We now retrace~\cite[Section 8]{HW} closely, with the metric case in mind.

    \begin{definition}[Interleavings, cross product ({\cite[Definition 20]{HW}})]
        Fix $n,l\in \NN $ and let $k=n+l$, and write $[k]$ for the set $\{0,\dots ,k\}$.

        Call a map $\sigma  = \langle \sigma _h,\sigma _v\rangle :[n+l] \rightarrow  [n]\times [l]$ satisfying
        \begin{itemize}
            \item $\sigma (0)=(0,0)$ and $\sigma (n+l)=(n,l)$;
            \item if $\sigma (i) = (a,b)$, then $\sigma (i+1)$ is either $(a+1,b)$ or $(a,b+1)$,
        \end{itemize}
        a staircase path.
        Write $\st{n,l}$ for the set of $(n,l)$ staircase paths.
        A staircase path is just a geodesic from $(0,0)$ to $(n,l)$ in the obvious grid.
        The sign $\sgn \sigma $ of $\sigma $ is $(-1)^s$, where $s$ is the number of squares “below the staircase”. 

        If $\xx$ is an $n$-\ppath{} in $X$, $\yy$ an $l$-\ppath{} in $Y$, seen as maps $[n] \rightarrow  X,[l] \rightarrow  Y$ (which we will always do) and $\sigma $ a staircase path, the interleaving of $\xx$ and $\yy$ along $\sigma $ is the $k$-\ppath{} $\xx \ovs{\sigma }{\times } \yy$ defined by $ \xx\ovs{\sigma }{\times }\yy := (\xx\times \yy)\circ \sigma $ (with $\xx\times \yy$ seen as a map  $[n]\times [l] \rightarrow  X\times Y$).

        The \emph{cross product}\footnote{but it's a square!} is the morphism of chain complexes:
        \[
            \square : \MC_*(X) \otimes  \MC_*(Y) \rightarrow  \MC_*(X\times Y),
        \]
        sending a tensor $\xx\otimes \yy$ to the alternating sum of its possible interleavings:
        \[
            \xx\otimes \yy \ \ovs{\square }{\mapsto } \sum _{\sigma \in \st{n,l}} \sgn(\sigma ) (\xx \ovs{\sigma }{\times } \yy).
        \]
    \end{definition}
 
    As a token of good will, let us check that $\square $ indeed defines a morphism of chain complexes.
    First, we need some terminology.
    We visualise a staircase path as an actual (irregular) staircase on the $[n]\times [l]$ grid, going from bottom-left $(0,0)$ to top-right $(n,l)$, with horizontal coordinate given by $\xx$  and vertical by $\yy$.
    Then, any $0<m<n+l$ defines a point $\sigma (m)$ on the staircase; exactly one of:
    \begin{description}
        \item[A corner:]
            which means that its predecessor and successor differ in both coordinates.
            There are two distinct types of corners, looking like $\ulcorner $ and $\lrcorner $ respectively.
        \item[A flat:]
            which means that $\sigma _v(m+1) = \sigma _v(m) = \sigma _v(m-1)$ and $\sigma _h(m) = \sigma _h(m-1)+1 = \sigma _h(m+1)-1$.
            Equivalently, $\sigma _v(m)$ has unique preimage $m$.
        \item[A wall:]
            which means that $\sigma _h(m+1) = \sigma _h(m) = \sigma _h(m-1)$ and $\sigma _v(m) = \sigma _v(m-1)+1 = \sigma _v(m+1)-1$.
            Equivalently, $\sigma _h(m)$ has unique preimage $m$.
    \end{description}
    
    Fix an $n$-\ppath{} $\xx$ and an $l$-\ppath{} $\yy$.
    For any index $m$, the value $\partial _m (\xx\ovs{\sigma }{\times }\yy)$ of the interleaving depends on the type of $m$.
    \begin{description}
        \item[$m$ is a corner:]
            then one easily sees that $(\xx\ovs{\sigma }{\times }\yy)_{m}$ is always between $(\xx\ovs{\sigma }{\times }\yy)_{m-1}$ and $(\xx\ovs{\sigma }{\times }\yy)_{m+1}$, but we do not need to know more than that.
        \item[$m$ is a flat:]
            then $(\xx\ovs{\sigma }{\times }\yy)_{m}$ is between $(\xx\ovs{\sigma }{\times }\yy)_{m-1}$ and $(\xx\ovs{\sigma }{\times }\yy)_{m+1}$ iff $\xx\circ \sigma _h(m)$ is between $\xx\circ \sigma _h(m-1)$ and $\xx\circ \sigma _h(m+1)$.
        \item[$m$ is a wall:]
            then $(\xx\ovs{\sigma }{\times }\yy)_{m}$ is between $(\xx\ovs{\sigma }{\times }\yy)_{m-1}$ and $(\xx\ovs{\sigma }{\times }\yy)_{m+1}$ iff $\yy\circ \sigma _v(m)$ is between $\yy\circ \sigma _v(m-1)$ and $\yy\circ \sigma _v(m+1)$.
    \end{description}
   
    If $m$ is a flat of $\sigma $, one can delete the column with coordinate $\sigma _h(m)$ and get a new staircase $\sigma _{\widehat m}$ in $\st{k-1,l}$.
    The sign of $\sigma _{\widehat m}$ differs from that of $\sigma $ by $(-1)^{\sigma _v(m)}$.
    By the above discussion, one concludes that:
    \begin{alignat*}{10}
        \partial _m (\xx \ovs{\sigma }{\times } \yy) &= \partial _{\sigma _h(m)}\xx \ovs{\sigma _{\widehat m}}{\times } \yy,\\
    \intertext{so that, using $\sigma _v(m)+\sigma _h(m)=m$:}
        (-1)^m \sgn(\sigma ) \partial _m (\xx \ovs{\sigma }{\times } \yy) &= (-1)^{\sigma _h(m)}\sgn(\sigma _{\widehat m}) \partial _{\sigma _h(m)}\xx \ovs{\sigma _{\widehat m}}{\times } \yy.
    \intertext{
        Similarly, if $m$ is a wall, $\sigma _{\widehat m}$ is obtained by deleting the row with coordinate $\sigma _v(m)$, the signs differ by $(-1)^{n-\sigma _h(m)}$, and we obtain:
    }
        (-1)^m \sgn(\sigma ) \partial _m (\xx \ovs{\sigma }{\times } \yy) &= (-1)^{n-\sigma _v(m)}\sgn(\sigma _{\widehat m})  \xx\ovs{\sigma _{\widehat m}}{\times }\partial _{\sigma _v(m)}\yy.
    \end{alignat*}

    Note that whenever we have a flat, the horizontal coordinate changes in two consecutive positions. 
    Hence flats are characterised by the positions where $0<i<n$ has a unique preimage under the map $\sigma_h$.
    Using this knowledge, we can define the following sets:
    \begin{alignat*}{10}
        S_h &:= \{(\sigma ,i)\ :\ \sigma \in \st{n,l} \text{ and $\sigma _h^{-1}(i)$ is a well-defined element}\},\\
    \intertext{and}
        S_v &:= \{(\sigma ,j)\ :\ \sigma \in \st{n,l} \text{ and $\sigma _v^{-1}(j)$ is a well-defined element}\}.
    \end{alignat*}
    Then, one easily sees that the map
    \begin{alignat*}{10}
        \st{n-1,l}\times \{1,\dots ,n-1\} \rightarrow  S_h
    \end{alignat*}
    defined by mapping $(\tau,i)$ to $(\sigma,i)$, where $\sigma$ is obtained by inserting a flat at position $\min \tau_h^{-1}(i)$, is bijective. 
    Its inverse is given by $(\sigma ,i) \mapsto (\sigma _{\widehat{\sigma_h^{-1}(i)}},i)$; similarly for
    \begin{alignat*}{10}
        \st{n,l-1}\times \{1,\dots ,l-1\} \rightarrow  S_v.
    \end{alignat*}

    By definition,
    \begin{alignat*}{10}
        \partial  (\xx \square  \yy)    &= \partial  \sum _{\sigma \in \st{n,l}} \sgn(\sigma ) (\xx\ovs{\sigma }{\times }\yy) \\
                         &= \sum _{\sigma \in \st{n,l}} \sgn(\sigma ) \sum _{m=1}^{n+l-1} (-1)^m \partial _m(\xx\ovs{\sigma }{\times }\yy).
    \intertext{
        Consider a given staircase path $\sigma $ and assume index $m$ of $\sigma $ is a corner, say $\ulcorner $.
        Then, there is a unique staircase path $\sigma _m'$ that is equal to $\sigma $ everywhere except at $m$, where it is the opposite corner, say $\lrcorner $.
        Clearly $\sigma $ and $\sigma _m'$ have opposite sign, and thus will cancel out when mapped through $\partial _m$.
        It follows that we can restrict the sum to indices which are not corners:
    }
                        &= \sum _{\sigma \in \st{n,l}} \sgn(\sigma ) \sum _{\substack{m=1\\ \text{not corner}}}^{n+l-1} (-1)^m \partial _m(\xx\ovs{\sigma }{\times }\yy)\\
    \intertext{ and since the remaining indices are either flats or walls:}
                        &= \sum _{\sigma \in \st{n,l}} \sgn(\sigma ) \left(  \sum _{\substack{m\in \{1,\dots ,n+l-1\} \\ \text{flat}}} (-1)^m \partial _m(\xx\ovs{\sigma }{\times }\yy) + \sum _{\substack{m\in \{1,\dots ,n+l-1\} \\ \text{wall}}} (-1)^m \partial _m(\xx\ovs{\sigma }{\times }\yy)  \right)  \\
                        &= \sum _{\sigma ,i\in S_h} (-1)^{\sigma _h^{-1}(i)} \sgn(\sigma ) \partial _{\sigma _h^{-1}(i)}(\xx\ovs{\sigma }{\times }\yy) + \sum _{\sigma ,j\in S_v} (-1)^{n+\sigma _v^{-1}(j)} \sgn(\sigma ) \partial _{\sigma _v^{-1}(j)}(\xx\ovs{\sigma }{\times }\yy)\\
                        &= \sum _{\sigma ,i\in S_h} (-1)^{i} \sgn(\sigma_{\widehat{\sigma_{h}^{-1}(i)}}) \partial _{i}\xx\ovs{\sigma_{\widehat{\sigma_{h}^{-1}(i)}}}{\times }\yy + \sum _{\sigma ,j\in S_v} (-1)^{j} \sgn(\sigma_{\widehat{\sigma_{v}^{-1}(j)}}) \partial _{j}\xx\ovs{\sigma_{\widehat{\sigma_{v}^{-1}(j)}}}{\times }\yy\\
    \intertext{and by the decomposition of staircases discussed above:}
                        &= \sum _{i=1}^{n-1} \sum _{\tau \in \st{n-1,l}} (-1)^i \sgn(\tau ) (\partial _i\xx)\ovs{\tau }{\times }\yy + \sum _{j=1}^{l-1} \sum _{\theta \in \st{n,l-1}} (-1)^{n+j} \sgn(\theta )  \xx\ovs{\theta }{\times }\partial _j\yy.\\
    \intertext{
        By definition, we also have
    }
     \square  (\partial (\xx\otimes \yy))      &=  \square  \left(\partial \xx\otimes \yy + (-1)^{n}\xx\otimes \partial \yy\right) \\
                                &= \sum _{i=1}^{n-1} (-1)^i \partial _i\xx \square  \yy + (-1)^{n}\sum _{j=1}^{l-1} (-1)^j\xx \square  \partial _j\yy\\                     &= \sum _{i=1}^{n-1} \sum _{\tau \in \st{n-1,l}} (-1)^i \sgn(\tau ) (\partial _i\xx\ovs{\tau }{\times }\yy) + \sum _{j=1}^{l-1} \sum _{\theta \in \st{n,l-1}} (-1)^{n+j}\sgn(\theta ) (\xx\ovs{\theta }{\times }\partial _j\yy)
    \end{alignat*}
    so that $\square  (\partial (\xx\otimes \yy)) = \partial  (\xx \square  \yy)$ and we have a chain map!

    \begin{proposition}[Künneth theorem ({\cite[Theorem 21]{HW}})]\label{proposition:Kunneth}
        The cross product induces a morphism
        \begin{alignat*}{10}
            \MH_*(X)\otimes \MH_*(Y) &\ \ovs{ \square }{\rightarrow }\ \MH_*(X\times Y)\\
                [f]\otimes [g]         &\mapsto  [f \square  g]
        \end{alignat*}
        which fits into a natural short exact sequence:
        \[
            0 \rightarrow  \MH_*(X)\otimes \MH_*(Y)\ \ovs{ \square }{\rightarrow }\ \MH_*(X\times Y) \rightarrow  \mathrm{Tor}(\MH_{*-1}(X),\MH_*(Y)) \rightarrow  0.
        \]
    \end{proposition}
    
    \begin{definition}[{\cite[Definition 40]{HW}}]\label{definition:metric_magnitude_sset}
        If $X$ is a metric space, we define the pointed simplicial set $\MS(X)$ as having $k$-simplices the $(k+1)$-tuples of points $\langle x_0,\dots ,x_{k}\rangle :[k]\rightarrow X$ in $X$, plus basepoint simplices $\pt_n$, along with face and degeneracy maps defined by:
        \begin{alignat*}{10}
            \ddd_{k,i} \langle x_0,\dots ,x_i,\dots ,x_k\rangle  &:= \begin{cases}
                \langle x_0,\dots ,\widehat{x_i},\dots ,x_k\rangle  &\text{if $x_i \in  [x_{i-1},x_{i+1}]$ }\\
                \pt_{k-1}                 &\text{otherwise,}\\
            \end{cases}
            \intertext{and}
                \sss_{k,i} \langle x_0,\dots ,x_i,\dots ,x_k\rangle  &:= \langle x_0,\dots ,x_i,x_i,\dots ,x_k\rangle ,\\
            \intertext{and on basepoints:}
                \ddd_{k,i} \pt_k &:= \pt_{k-1}\\
                \sss_{k,i} \pt_k &:= \pt_{k+1}.
        \end{alignat*}
    \end{definition}

    \begin{proposition}[{\cite[Proposition 41]{HW}}]
        Let $X,Y$ metric spaces.
        The following morphism of pointed simplicial sets:
        \begin{alignat*}{10}
             \square : \MS(X) \wedge  \MS(Y) &\rightarrow   \MS(X\times Y)\\
             [\xx,\yy] &\mapsto  \langle \xx,\yy\rangle 
        \end{alignat*}
        is an isomorphism. 
    \end{proposition}
    Let us clarify the notation: $\xx$ and $\yy$ are $n$-simplices of $\MS(X)$ and $\MS(Y)$ respectively, that is, maps $[n]\rightarrow X$ and $[n]\rightarrow Y$ respectively.
    Thus, the pair $(\xx,\yy)$ is an element of $\MS(X)\times \MS(Y)$, and $[\xx,\yy]$ an element of $\MS(X)\wedge \MS(Y)$.
    Finally, $\langle \xx,\yy\rangle :[n] \rightarrow  X\times Y$ is the “product” of the given maps, hence an element of $\MS(X\times Y)$.
    \begin{proof}
        Bijectivity and commutation with degeneracy maps is clear.
        For face maps, one uses the product identity for intervals in the $l^{1}$ product, plus the fact that $\ddd_{k,i}(\xx,\yy) \neq  \pt_{k-1}$ iff both $\ddd_{k,i}\xx\neq \pt_{k-1}$ and $\ddd_{k,i}\yy\neq \pt_{k-1}$ hold. 
    \end{proof}

    Still following~\cite{HW}, given a simplicial set $S$, the \emph{normalised reduced} chain complex $\NR_*(S)$ associated to $X$ is defined by:
    \[
        \NR_k(S) := \ZZ\{\text{$k$-simplices}\}/\ZZ\{\text{degenerate and basepoint simplices}\},
    \]
    with boundary map induced by:
    \[
        \partial _k = \sum _{i=0}^k (-1)^i \ddd_{k,i}.
    \]

    Since a simplex in $\MS(X)$ is degenerate iff it has consecutive equal points, the following holds:

    \begin{proposition}[{\cite[Lemma 42]{HW}}]
        $\NR_k(\MS(X))$ and $\MC_k(X)$ are isomorphic chain complexes.
    \end{proposition}
    \begin{proof}
        $\MC_k(X)$ is generated by the $k$-\ppath{}s in $X$; that is, the $(k+1)$-tuples of consecutively distinct points in $X$.
        $\NR_k(\MS(X))$ is generated by the non-degenerate non-basepoint simplices of $\MS(X)$ which are exactly the $k$-\ppath{}s.
        Thus, the groups are isomorphic.
        On $\NR_k(\MS(X))$, the boundary is defined as $\partial _k = \sum _{i=1}^{k-1}(-1)^i \ddd_{k,i}$, and since $\ddd_{k,i}$ sends a simplex $\xx$ to a basepoint iff $x_i \notin  [x_{i-1},x_{i+1}]$, $\ddd_{k,i}$ sends $\xx$ to zero at the level of chain maps, which shows that the boundary maps agree.
    \end{proof}
    From now on, we will identify $\NR_k(\MS(X))$ with $\MC_k(X)$.

    Remember that given a simplicial set $S$, there exists, for any $n$, a natural bijection:
    \[
        \facemorphism{\bullet }:\Mor(\Delta [n],S) \leftrightarrow  S_n,
    \]
    where $\Delta [n]$ is the canonical $n$-simplex.
    This bijection is obtained by sending a morphism $f:\Delta [n] \rightarrow  S$ to the image through $f$ of the single non-degenerate $n$-simplex $\Id:[n]\rightarrow [n]$ in $\Delta [n]$.

    If $\sigma =\langle \sigma _h,\sigma _v\rangle $ is a $(k,l)$-staircase, $\sigma $ defines a morphism of simplicial complexes:
    \[
        \sigma _*: \Delta [k+l] \rightarrow  \Delta [k]\times \Delta [l],
    \]
    by sending a face $f:[n] \rightarrow  [k+l]$ of $\Delta [k+l]$ to the pair of faces $(\sigma _h\circ f:[n] \rightarrow  [k],\sigma _v\circ f:[n]\rightarrow [l])$ in $\Delta [k ]\times \Delta [l]$.
    Finally, if $\xx$ and $\yy$ are simplices in $S$ and $T$ respectively, they are naturally associated to morphisms $\facemorphism{\xx}:\Delta [k]\rightarrow S,\facemorphism{\yy}: \Delta [l]\rightarrow T$, so that one has a morphism:
    \[
        \facemorphism{\xx}\times \facemorphism{\yy}:\Delta [k]\times \Delta [l] \rightarrow  S\times T.
    \]

    Given pointed simplicial sets $S,T$, we now define the \emph{reduced Eilenberg-Zilber map}.
    \begin{alignat*}{10}
        \REZ: \NR_*(S) \otimes  \NR_*(T)                   &\rightarrow       \NR_*(S\wedge T)\\
        \xx\otimes \yy  \in \NR_{k}(S)\otimes \NR_{l}(T)               &\mapsto       \sum _{\sigma \in \st{k,l}} [\morphismface{(\facemorphism{\xx}\times \facemorphism{\yy})\circ \sigma _*}]
    \end{alignat*}
    where $[\bullet ]:S\times T \rightarrow  S\wedge T$ is the collapsing map.

    The following abstract property of $\REZ$ is proved in~\cite{HW}:
    \begin{proposition}[{\cite[Proposition 43]{HW}}]
        $\REZ$ is a quasi-isomorphism.
    \end{proposition}

    Let us now concretely describe the map $\REZ$ in the case at hand:
    Fix generators $\xx \in  \NR_k(\MS(X))$ and $\yy \in  \NR_l(\MS(Y))$.
    Seen as a simplex of $\MS(X)$, we have
    \begin{alignat*}{10}
        \xx: [k] &\rightarrow  X\\
    \intertext{and through the identification “simplex”$\leftrightarrow $“morphism”, as }
        \facemorphism{\xx}: \Delta [k]                   &\rightarrow  \MS(X)\\
        (\phi :[n]\rightarrow [k]) \in  \Delta [k]_n    &\mapsto  (\xx\circ \phi :[n] \rightarrow  X) \in  \MS(X)_n,
    \end{alignat*}
    and similarly for $\yy$.
    Thus, the composite 
    \begin{alignat*}{10}
        (\facemorphism{\xx}\times \facemorphism{\yy})\circ \sigma _* : \Delta [k+l]  &\rightarrow  \MS(X)\times \MS(Y)\\
    \intertext{is defined as}
        (\phi :[n] \rightarrow  [k]) \in  \Delta [k+l]_n       &\mapsto  ((\xx_*\times \yy_*)\circ \sigma _*)\phi  \in  \MS(X)\times \MS(Y)\\
                                       &= (\xx\circ \sigma _h\circ \phi ,\yy\circ \sigma _v\circ \phi ),\\
    \intertext{and passing back from morphism to simplices (evaluating at $\Id:[k+l]\rightarrow [k+l]$), the result is simply}
        (\xx\circ \sigma _h\circ \Id,\yy\circ \sigma _v\circ \Id) = (\xx\circ \sigma _h,\yy\circ \sigma _v)&\in  \MS(X)_k\times \MS(Y)_l.
    \end{alignat*}

    \begin{proposition}[{\cite[Proof of Theorem 21]{HW}}]
        The cross product 
        \[
             \square : \MC_*(X)\otimes \MC_*(Y) \rightarrow  \MC_*(X\times Y)
        \]
        is a quasi-isomorphism.
    \end{proposition}
    \begin{proof}
        We consider the following chain of maps:
        \begin{alignat*}{10}
            \MC_*(X)\otimes \MC_*(Y)\ &\ovs{\phantom{\mathclap{\NR_*(\square )}}}{\cong } \quad                      \NR_*(\MS(X)) \otimes  \NR_*(\MS(Y)) \\
                               &\ovs{\mathclap{\REZ}}{\rightarrow } \quad        \NR_*(\MS(X) \wedge  \MS(Y)) \\
                               &\ovs{\mathclap{\NR_*(\square )}}{\cong } \quad    \NR_*(\MS(X\times Y)) \\
                               &\ovs{\phantom{\mathclap{\NR_*(\square )}}}{\cong } \quad                    \MC_*(X\times Y).
        \end{alignat*}
        All but $\REZ$ are isomorphisms, and $\REZ$ is a quasi-isomorphism; the composite is therefore a quasi-isomorphism.
        By following along a generator $\xx\otimes \yy \in  \MC_k(X)\otimes \MC_l(Y)$, one has
        \begin{alignat*}{10}
            \xx\otimes \yy \ &\ovs{\mathclap{=}}{\mapsto }& \quad              &\xx\otimes \yy\\
                      &\ovs{\mathclap{\REZ}}{\mapsto }& \quad           &\sum _{\sigma \in \st{k,l}} \sgn(\sigma )[(\xx\circ \sigma _h,\yy\circ \sigma _v)]\\
                      &\ovs{\mathclap{\NR_{k+l}(\square )}}{\mapsto }& \quad   &\sum _{\sigma \in \st{k,l}} \sgn(\sigma )[(\xx\times \yy)\circ \sigma ]\\
                      &\ovs{\mathclap{=}}{\mapsto }&  \quad             &\sum _{\sigma \in \st{k,l}} \sgn(\sigma )[(\xx\times \yy)\circ \sigma ] = \xx\square \yy.
        \end{alignat*}
        and the composite is really just the cross product.
    \end{proof}
   
    \begin{proof}[Proof of~\Cref{proposition:Kunneth}]
        Applying the algebraic Künneth formula to $\MC_*(X)$ and $\MC_*(Y)$ yields a short exact sequence
        \[
            0 \rightarrow  \MH_{*}(X) \otimes  \MH_{*}(Y) \rightarrow  H_*(\MC_*(X)\otimes \MC_*(Y)) \rightarrow  \mathrm{Tor}(\MH_{*-1}(X),\MH_{*}(Y)) \rightarrow  0,
        \]
        and by the above, the middle term is isomorphic, through $H_*(\square )$, to $\MH_k(X\times Y)$.
        Naturality follows from naturality in the algebraic Künneth formula and that of the cross product.
    \end{proof}
    
    Note that the “length aware” sequence in~\cite{HW} can easily be recovered by fixing $l$ in $H_*(\MC_*(X)\otimes \MC_*(Y))$.

    \subsection{Excision}

    \begin{definition}[Gated sets]
        Given a metric space $X$, a subset $A$ of $X$ is said to be \emph{gated} if for any $x\in X$, there exists some $a_x\in A$ such that $a_x$ is between $x$ and any $a\in A$.
        The point $a_x$ is called a \emph{gate} between $x$ and $A$.
    \end{definition}
    
    Gated sets enjoy the following properties (see~\cite{DS}).
    \begin{proposition}[{\cite[pp. 114,112, 115 respectively]{DS}}]
    \begin{itemize}
        \item
            Gated sets are convex.
        \item
            For any $x$ and gated $A$, there exists a unique gate $a_x$.
        \item
            The map $x\mapsto a_x$ is non-expanding, and is the identity on $A$.
    \end{itemize}
    \end{proposition}
    From now on, we write $\pi :X\rightarrow A$ for the map sending $x$ to $a_x$.
    Note that by the above, $\pi $ is a retraction from $X$ to $A$.

    \paragraph{Summary of differences.}
    As in the proof of the Künneth formula, the proofs of excision and Mayer-Vietoris in~\cite{HW} essentially generalise without trouble, yet some care is needed.
    This is mainly because the notions of convexity and subgraphs used in~\cite{HW} do not exactly match ours.
    For instance, it is possible for a gated decomposition $X=Y\cup Z$ to not be valid in the sense of~\cite{HW} because of the existence of edges connecting $Y$ and $Z$.
    Conversely, a subgraph may be convex in the sense of~\cite{HW} but not in the sense used here.
    Since the notions used in~\cite{HW} are less natural in the metric case, we chose not to expand on them.

    Apart from definitional differences, the main obstacle to generalising \cite{HW} comes in their \cite[Lemma 51]{HW} and \cite[Proof of Theorem 29]{HW} in which, once a length $l$ is fixed, they use the vanishing of the groups $\MC_k^l(X)$ for $k>l$; this does not hold in general for metric spaces.
    Thus, the vanishing of the homologies of the respective quotients in their arguments can no longer be proven for all $k$ at once.

    Finally, unlike~\cite{HW}, we have not not verified naturality of the Mayer-Vietoris sequence in the metric case.

    To conclude, we will (again) follow~\cite[Section 9]{HW} very closely and make changes as needed.
    
    \begin{definition}[Gated decomposition]
        If $X$ is a metric space and $Y,Z,W$ are subspaces satisfying $X=Y\cup Z,\ W=Y\cap Z$ and $W$ is gated w.r.t. $Z$, then we say that the triple $(X,Y,Z)$ is a \emph{gated} decomposition. 
    \end{definition}

    Following~\cite{HW}, we write $\MC_*(Y,Z)$ for the subcomplex of $\MC_*(X)$ spanned by \ppath{}s entirely contained in either $Y$ or $Z$.
    We can now state the Excision theorem:
    \begin{theorem}[Excision – metric setting]\label{theorem:metric_excision}
        If $X=Y\cup Z, W=Y\cap Z$ is a gated decomposition of $X$, then the inclusion
        \[
            \MC_*(Y,Z) \hookrightarrow  \MC_*(X)
        \]
        is a quasi-isomorphism.
    \end{theorem}

    Once one has excision, Mayer-Vietoris follows easily:

    \begin{theorem}[Mayer-Vietoris – metric case]
        If $X=Y\cup Z, W=Y\cap Z$ is a gated decomposition of $X$, then the inclusions 
        \[
            j_Y:W\rightarrow Y,\ j_Z:W\rightarrow Z,\ i_Y:Y\rightarrow X,\ i_Z:Z\rightarrow X
        \] 
        induce a short exact sequence:
        \[
            0 \rightarrow  \MH_*(W) \ovs{\langle (j_Y)_*, -(j_Z)_*\rangle}{\longrightarrow } \MH_*(Y)\oplus \MH_*(Z) \ovs{(i_Y)_*\oplus (i_Z)_*}{\longrightarrow } \MH_*(X) \rightarrow  0.
        \]
    \end{theorem}
    The proof follows~\cite[Proof of Theorem 29, assuming Theorem 28]{HW}.
    \begin{proof}
        We have a short exact sequence of chain complexes
        \[
            0 \rightarrow  \MC_*(W) \ovs{\langle (j_Y)_*, -(j_Z)_*\rangle}{\longrightarrow } \MC_*(Y)\oplus \MC_*(Z) \ovs{(i_Y)_*\oplus (i_Z)_*}{\longrightarrow } \MC_*(Y,Z) \rightarrow  0,
        \]
        which induce a long exact sequence in homology:
        \[
            \dots  \rightarrow  \MH_*(W) \ovs{\langle (j_Y)_*, -(j_Z)_*\rangle}{\longrightarrow } \MH_*(Y)\oplus \MH_*(Z) \ovs{(i_Y)_*\oplus (i_Z)_*}{\longrightarrow } \MH_*(Y,Z) \rightarrow  \MH_{*-1}(W) \rightarrow  \dots 
        \]
        and we will verify that $\langle (j_Y)_*, -(j_Z)_*\rangle$ is a monomorphism, which implies that the sequence decomposes into fragments:
        \[
            0 \rightarrow  \MH_*(W) \ovs{\langle (j_Y)_*, -(j_Z)_*\rangle}{\longrightarrow } \MH_*(Y)\oplus \MH_*(Z) \ovs{(i_Y)_*\oplus (i_Z)_*}{\longrightarrow } \MH_*(Y,Z) \rightarrow  0, 
        \]
        and composing with the excision isomorphism yields the desired short exact sequence.

        Since $\pi :Z\rightarrow W$ is a retraction, $\pi _*:\MH_*(Z)\rightarrow \MH_*(W)$ is an epimorphism, and so is the composite
        \[
            \MH_*(Y)\oplus \MH_*(Z) \ovs{\prj_2}{\rightarrow } \MH_*(Z) \ovs{\pi _*}{\rightarrow } \MH_*(W) 
        \]
        and $\langle (j_Y)_*, -(j_Z)_*\rangle$ serves as a left inverse, hence is injective.

    \end{proof}

    \subsubsection{Proof of Excision}
    \newcommand{\tB}{\tilde B}

    Let us fix for the remainder of the section a gated decomposition $X=Y\cup Z,W=Y\cap Z$. 

    We define, for $a\in Y-Z,b\in Z-Y$ (or vice versa):
    \begin{alignat*}{10}
        A_*(a,b) &:= \spn \left\{ \langle x_0,\dots ,x_k\rangle  \ \middle|\ x_0=a,x_k=b, x_1,\dots ,x_{k-1}\in W \right\} \leq  \MC_*(X).\\
    \intertext{For $b\in Z-Y$:}
        B_*(b) &:= \spn \left\{ \langle x_0,\dots ,x_k\rangle  \ \middle|\ x_k=b, x_0,\dots ,x_{k-1}\in Y  \right\} \leq  \MC_*(X),\\
        \tB_*(b) &:= \spn \left\{ \langle x_0,\dots ,x_k\rangle  \ \middle|\ x_k=b, x_0,\dots ,x_{k-1}\in W  \right\} \leq  \MC_*(X),\\
    \intertext{and $i\in \NN$:}    
        F_*(b;i) &:= \spn \left\{ \langle x_0,\dots ,x_k\rangle  \ \middle|\ x_k=b, x_0,\dots ,x_{k-1}\in Y, x_i,\dots ,x_{k-1}\in Z \right\} \leq  \MC_*(X),\\
    \intertext{and symmetrically for $b\in Y-Z$. Finally for $i\in \NN$:}
        G_*(i) &:= \spn \left\{ \langle x_0,\dots ,x_k\rangle  \ \middle|\ x_0,\dots ,x_{k-i}\text{ all lie in $Y$, or all lie in $Z$} \right\} \leq  \MC_*(X).
    \end{alignat*}
    It is clear that $G_*(0) = \MC_*(Y,Z)$, $G_k(l) = \MC_k(X)$ for all $k\leq l$, and $G_*(l)\leq G_*(l+1)$ for all $l$.
    It follows that $\MC_*(X)$ is the direct limit of the system
    \[
        \MC_*(Y,Z) = G_*(0) \leq  G_*(1) \leq  \dots  \leq  G_*(l) \leq  G_*(l+1) \leq  \dots .
    \]
    
    Thus, to show that the inclusion $\MC_*(Y,Z) \hookrightarrow  \MC_*(X)$ is a quasi-isomorphism, it is enough to do so for each inclusion $G_*(l) \hookrightarrow  G_*(l+1)$.
    Indeed, once this is done, the whole system becomes a chain of isomorphisms after getting homologised, and each inclusion to the limit $\MC_*(X)$ as well.
    In particular, so does the inclusion $\MC_*(Y,Z) \hookrightarrow  \MC_*(X)$.

    This is essentially the only thing we have to change from the argument of~\cite{HW}.

    Let $l$ be fixed from now on, and given a chain complex $C_*$, write $\Sigma ^jC_*$ for the shifted chain complex $(\Sigma ^jC_*)_k = C_{k-j}$.
        
    \begin{proposition}[{\cite[Lemma 48]{HW}}]\label{proposition:A_is_acyclic}
        The complex $A_*(a,b)$ is acyclic.
    \end{proposition}
    \begin{proof}
        Assume that $a\in Y-Z,b\in Z-Y$ (the other case is treated symmetrically).
        We construct a chain homotopy $s$ between $\Id$ and $0$ on $A_*(a,b)$ as follows:
        \begin{alignat*}{10}
            s_k: A_k(a,b) &\rightarrow A_{k+1}(a,b)\\
                 (a,x_1,\dots ,x_{k-1},b) &\mapsto  
                 \begin{cases}
                     (-1)^k (a,x_1,\dots ,x_{k-1},\pi (b),b) & \text{ if $\pi (b)\neq x_{k-1}$}\\
                    0                       & \text{ otherwise,}
                 \end{cases}
        \end{alignat*}
        where $\pi :Z\rightarrow W$ is the gate projection.
        
        Thus, we have to check that $\partial \circ s - s\circ \partial  = \Id$.
        Let us fix a \ppath{} $\xx = \langle a,x_1,\dots ,x_{k-1},b\rangle $.
        Recall that the boundary map $\partial _k:\MC_k(X) \rightarrow  \MC_{k-1}(X)$ is defined as the alternating sum of the maps $\partial _{k,i}$ which drop the index $i$ iff it lies between its neighbours. 
        Since $s$ behaves independently of the content of the \ppath{} up to $x_{k-1}$ and, for $i\leq k-2$, each $\partial _{k,i}$ behaves independently of the content after $x_{k-1}$, it follows that $\partial _{k+1,i}s - s\partial _{k,i} = 0$ for $i\leq k-2$.
        Thus, it remains to verify that
        \begin{equation}\label{equation:A_is_acyclic}
            (-1)^{k-1}\partial _{k+1,k-1}\circ s +  (-1)^{k}\partial _{k+1,k}\circ s + (-1)^{k-1}s\circ \partial _{k,k-1} = \Id.
        \end{equation}

        By definition of a gated set, and since $x_{i} \in  W$ for all $1\leq i\leq k-1$ we know that $\pi (b) \in  [b,x_{i}]$ holds.
        We now verify the equation by distinguishing cases:

        \begin{description}
            \item[If $x_{k-1}=\pi (b)$:]
                Then $s\xx=0$, by definition so the two leftmost terms of~\Cref{equation:A_is_acyclic} vanish; and $x_{k-1} \in  [x_{k-2},b]$, so that $\partial _{k,k-1}\xx = \langle a,\dots ,x_{k-2},b\rangle $.
                Since $x_{k-2}\neq x_{k-1}=\pi (b)$, $s\langle a,\dots ,x_{k-2},b\rangle $ is $(-1)^{k-1}\langle a,\dots ,x_{k-1},\pi (b),b\rangle  = (-1)^{k-1}\xx$, the rightmost term is $\xx$ and the equation is verified.

            \item[If $x_{k-1}\neq \pi (b)$:]
                Then $s\xx = (-1)^k\langle a,x_1,\dots ,x_{k-2},x_{k-1},\pi (b),b\rangle $, and  $\partial _{k+1,k}s\xx \allowbreak = (-1)^k\langle a,x_1,\dots ,x_{k-2},x_{k-1},b\rangle \allowbreak = (-1)^k\xx$; the middle term of~\Cref{equation:A_is_acyclic} is therefore $\xx$.
                
               \begin{description}
                   \item[{If moreover $x_{k-1}\in [x_{k-2},b]$ :}]  
               			Then $\partial_{k+1,k-1}\circ s\xx = (-1)^k\langle a,x_1,\dots ,x_{k-2},\pi (b),b\rangle $  and     $s\circ \partial _{k,k-1}\xx = (-1)^{k+1}\langle a,x_1,\dots ,x_{k-2},\pi (b),b\rangle $, hence they cancel out.  
                    \item[{If on the other hand $x_{k-1}\notin [x_{k-2},b]$ :}] 
               			Then $\partial _{k,k-1}\xx=0$ and since also  $x_{k-1} \notin [x_{k-2},\pi(b)]$, consequently $\partial_{k+1,k-1}\circ s\xx =0$.
               \end{description}
         \end{description}

    \end{proof}
    
    Define the set:
    \[
        J_Z(l) := \{\xx = \langle x_0,\dots ,x_{l}\rangle \ :\ x_0,\dots x_{l} \in  Y,\ x_l \notin Z\},
    \]
    and define $J_Y(l)$ symmetrically. 
    \begin{proposition}[{\cite[Lemma 51]{HW}}]\label{proposition:consecutive_Fs_are_quasi_isom}
        For any $b\in Z-Y$, we have an isomorphism:
        \[
            F_*(b,l+1)/F_*(b,l) \cong  \bigoplus _{\xx\in J_Z(l)} \Sigma ^{l} A_*(x_l,b).
        \]
        In particular, the quotient $F_*(b,l)/F_*(b,0)$ is acyclic.

        The same holds for $b\in Y-Z$ with $J_Y(l)$ instead of $J_Z(l)$.
    \end{proposition}
    \begin{proof}

        By definition of the groups $F_k(b,l+1)$ and $F_k(b,l)$, the quotient is spanned freely by the \ppath{}s $\xx = \langle x_0,\dots ,x_k\rangle $ satisfying:
        \begin{enumerate}[label=(\roman*),noitemsep]
            \item
                $x_k=b$;
            \item
                $x_0,\dots ,x_{k-1} \in  Y$;
            \item
                $x_{l+1},\dots ,x_{k-1} \in  Z$;
            \item
                $x_l \notin  Z$;
        \end{enumerate}   
        where the first three conditions stem from membership in $F_*(b,l+1)$, and the last from non-membership in $F_*(b,l)$.
        For $i\leq l$ such a generator $\langle x_0,\dots ,x_l,x_{l+1},\dots ,x_k\rangle $ is mapped by $\partial _i$ into $F_*(b,l)$ since $x_{l+1}$ is moved to index $l$ and lies in $Z$.
        Thus, $\partial _i$ becomes the zero map in the quotient. 
        For $i>l$, $\partial _i$ maps a generator into $F_*(b,l+1)$.
        Therefore, the boundary map on $F_*(b,l+1)/F_*(b,l)$ is:
        \[
            \sum _{i>l} (-1)^i \partial _i.
        \]

        Similarly, the complex $\bigoplus _{\xx\in J_Z(l)} \Sigma ^{l} A_*(x_l,b)$ is freely spanned by pairs $\langle x_0,\dots ,x_l\rangle ,\langle x_l,\dots ,x_k\rangle $ satisfying:
        \begin{enumerate}[label=(\roman*),noitemsep]
            \item
                $x_k=b$;
            \item
                $x_{l+1},\dots ,x_{k-1}\in  Y\cap Z$;
            \item
                $x_0,\dots ,x_l\in Y$
            \item
                $x_l\notin Z$;
        \end{enumerate}
        where the first two conditions stem from membership of $\langle x_l,\dots ,x_k\rangle $ in $\Sigma ^lA_*(x_l,b)$, and the last two from membership of $\langle x_0,\dots ,x_l\rangle $ in $J_Z(l)$.
        The boundary map on $\bigoplus _{\xx\in J_Z(l)} \Sigma ^{l} A_*(x_l,b)$ is clearly just 
        \[
            \partial : \langle x_0,\dots ,x_l\rangle ,\langle x_l,\dots ,x_k\rangle  \mapsto  \langle x_0,\dots ,x_l\rangle , \sum _{i=1}^{k-l} (-1)^i \partial _i\langle x_l,\dots ,x_k\rangle 
        \]
        which corresponds to applying $\sum _{i>l} (-1)^i \partial _i$ to the merged \ppath{} $\langle x_0,\dots ,x_k\rangle $.

        The isomorphism is then defined by:
        \begin{alignat*}{10}
            \phi :&& \bigoplus _{\xx\in J(l)} \Sigma ^{l} A_*(x_l,b)                     &\rightarrow  F_*(b,l+1)/F_*(b,l) \\
              &&\underbrace{\langle x_0,\dots ,x_l\rangle }_{\in  J_Z(l)}, \underbrace{x_l,\dots ,x_k}_{\in  A_{k-l}(x_l,b)}    &\mapsto  \langle x_0,\dots ,x_k\rangle ,
        \end{alignat*}
        and the above analysis of the two sides shows that $\phi $ is an isomorphism of chain complexes.
    \end{proof}
    
    \begin{proposition}
        For any $b\in Y\Delta Z$, the quotient $B_*(b)/\tB_*(b)$ is acyclic.
    \end{proposition}
    \begin{proof}
        Consider the directed system:
        \[
            \tB_*(b) = F_*(b,0) \leq  F_*(b,1) \leq  \dots  \leq  F_*(b,l) \leq  F_*(b,l+1) \leq  \dots 
        \]
        Since we have inclusions $F_*(b,l) \leq  B_*(b)$ for all $l$, and for each $k\leq l$,
        \[
            F_k(b,l) = B_k(b),
        \]
        it follows that $B_*(b)$ is the direct limit of the above system.
        Passing to homology, each inclusion $F_*(b,l)\leq F_*(b,l+1)$ becomes an isomorphism by~\Cref{proposition:consecutive_Fs_are_quasi_isom}.
        It follows then that the inclusion $\tB_*(b) \leq  B_*(b)$ also becomes an isomorphism.
    \end{proof}

    Define now the set:
    \[
        K(l) := \{\xx = \langle x_{k-l},\dots ,x_{k}\rangle \ :\ x_{k-l} \in  Y\Delta Z\}.
    \]
    \begin{proposition}[{\cite[Proof of Theorem 29]{HW}}]\label{proposition:each_step_is_acyclic}
        We have an isomorphism:
        \[
            G_*(l+1)/G_*(l) \cong  \bigoplus _{\xx\in K(l)} \Sigma ^{l} B_*(x_{k-l})/\tB_*(x_{k-l}).
        \]
        In particular, the inclusion $G_*(l) \leq  G_*(l+1)$ is a quasi-isomorphism.
    \end{proposition}
    \begin{proof}
        We proceed similarly to the proof of~\Cref{proposition:consecutive_Fs_are_quasi_isom}.

        The chain complex $G_*(l+1)/G_*(l)$ is spanned freely by the \ppath{}s  $\xx = \langle x_0,\dots x_{k-l-1},x_{k-l},\dots ,x_k\rangle $ satisfying either:
        \begin{enumerate}[label=(\roman*),noitemsep]
            \item 
                $x_0,\dots ,x_{k-l-1}$ all lie in $Y$;
            \item
                $x_0,\dots ,x_{k-l-1}$ do not all lie in $W$;
            \item
                $x_{k-l}$ does not lie in $Y$;
        \end{enumerate}
        or symmetrically with $Y$ replaced by $Z$.
        Indeed, the first condition stems from membership in $G_*(l+1)$, while the other two conditions from non-membership in $G_*(l)$: if $x_{k-l}\in Y$ held, we would clearly have $\xx \in  G_*(l)$, and if all $x_0,\dots ,x_{k-l-1}$ lay in $W$, they would also all lie in $Z$, and assuming $x_{k-l}\notin Y$, necessarily $x_{k-l}\in Z$ so that $\xx \in  G_*(l)$.

        For $i\geq k-l$, the image of such a generator under $\partial _i$ does not satisfy the last condition, hence is necessarily mapped to zero in the quotient.
        For $i\leq k-l-1$, its image under $\partial _i$ will be zero if the second condition becomes unsatisfied, and is “kept” otherwise.

        For $b\in Y\Delta Z$, the chain complex $\Sigma ^l B_*(b)/\tB_*(b)$ is spanned freely by the \ppath{}s $\xx = \langle x_0,\dots ,x_{k-l}\rangle $ satisfying: 
        \begin{enumerate}[label=(\roman*),noitemsep]
            \item
                $x_{k-l}=b$;
            \item
                $x_0,\dots ,x_{k-l-1} \in  Y$;
            \item
                there exists $0\leq i\leq k-l-1$ with $x_i \notin  Z$;
        \end{enumerate}
        if $b\in Z-Y$, and symmetrically if $b\in Y-Z$.
        The image of a generator of $\Sigma ^l B_*(b)/\tB_*(b)$ under $\partial _i$ will still satisfy the first two conditions, and thus be zero depending on whether the last condition becomes unsatisfied.

        Thus, the chain complex $\bigoplus _{\xx \in  K(l)} \Sigma ^{l} B_*(x_{k-l})/\tB_*(x_{k-l})$ is spanned by pairs of \ppath{}s  $\langle x_0,\dots ,x_{k-l}\rangle ,\langle x_{k-l},\dots ,x_k\rangle $ such that, in addition to the above conditions with $b:=x_{k-l}$ on the first \ppath{}, we have
        \begin{enumerate}[label=(\roman*),resume,noitemsep]
            \item
                $x_{k-l} \in  Y\Delta Z$.
        \end{enumerate}
        The boundary map on $\bigoplus _{\xx \in  K(l)} \Sigma ^{l} B_*(x_{k-l})/\tB_*(x_{k-l})$ is just the sum of the boundary maps on each $B_*(x_{k-l})/\tB_*(x_{k-l})$, which we observe having the same behaviour as on $G_*(l+1)/G_*(l)$.

        Thus, we define our isomorphism as:
        \begin{alignat*}{10}
            \psi : && \bigoplus _{\xx \in  K(l)} \Sigma ^{l} B_*(x_{k-l})/\tB_*(x_{k-l}) &\rightarrow  G_*(l+1)/G_*(l) \\
               &&\underbrace{\langle x_{k-l},\dots ,x_k\rangle }_{\in  K(l)}, \underbrace{\langle x_0,\dots ,x_{k-l}\rangle }_{\in  B_{k-l}(x_{k-l})/\tB_*(x_{k-l})}    &\mapsto  \langle x_0,\dots ,x_k\rangle .
        \end{alignat*}
        and having equal generating sets and agreeing boundary maps, this indeed is a chain complex isomorphism.
    \end{proof}

    \begin{proof}[Proof of {\cref{theorem:metric_excision}}]
        Each inclusion in the direct system
        \[
            \MC_*(Y,Z) = G_*(0) \leq  \dots  \leq  G_*(l) \leq  G_*(l+1) \leq  \dots  \dots  \leq  (\MC_*(X))
        \]
        is a quasi-isomorphism by~\Cref{proposition:each_step_is_acyclic}, and $\MC_*(X)$ is the direct limit of this system.
        Thus, the inclusions induce isomorphisms:
        \[
            \MH_*(Y,Z) = H(G_*(0)) \cong   \dots   \cong  H(G_*(l)) \cong  H(G_*),
        \]
        with $\MH_*(X)$ (along with the morphisms induced by inclusions into $\MC_*(X)$) their limit.
        It follows that each inclusion, among which $\MC_*(Y,Z)\hookrightarrow \MH_*(X)$ is a quasi-isomorphism.
    \end{proof}

    \section{Diagonality}\label{section:diagonality} 

    In the first section, we have seen that median graphs are diagonal (in the sense of Hepworth and Willerton).
    Knowing that median \emph{graphs} are special cases of median \emph{spaces} motivates us to try and find a corresponding description for median spaces.
    In this section, we introduce the notion of diagonality for metric spaces and verify some of its properties.
    As hoped, we will see in the next section that median spaces indeed are diagonal.
    This section and the next should make a strong case for this being a worthy generalisation of the original notion of diagonality.

    \begin{definition}[Diagonality]
        A space $X$ is said to be \emph{diagonal} if $\MH_k(X)$ is generated by chains of saturated \ppath{}s, for all $k,l$. 
    \end{definition}
    Let $S_k(X)$ denote the span of saturated \ppath{}s, as a submodule of $\MC_k(X)$.
    It is useful to remark that:
    \begin{proposition}\label{proposition:S_cap_B_is_trivial}
        The support of elements in $S_k(X)$ and $\MB_k(X)$ are disjoint.
        In particular, $S_k(X)\cap \MB_k(X) = \{0\}$.
    \end{proposition}
    \begin{proof}
        By the definition of the boundary operator, any \ppath{} in the boundary comes from a larger \ppath{}, hence is not saturated.
    \end{proof}
    
    Thus, diagonality can be restated in different ways:
    \begin{proposition}\label{proposition:diagonal_implies_homology_is_Z_cap_S}
        The following are equivalent:
        \begin{enumerate}
            \item\label{X_is_diago}
                $X$ is diagonal;
            \item\label{Z_eq_Z_cap_S_oplus_B}
                $\MZ_k(X) =  (\MZ_k(X) \cap  S_k(X)) \oplus  \MB_k(X)$;
            \item\label{decomp_sat}
                Given a cycle $\sigma  = \sigma _S + \sigma ' \in  \MZ_k(X)$, with $\sigma _S$ the part corresponding to saturated \ppath{}s and $\sigma '$ the rest, we must have $\sigma ' \in  \MB_k(X)$;
            \item\label{incl_then_quot_is_isom}
                The “inclusion-then-quotient” morphism $\MZ_k(X) \cap  S_k(X) \rightarrow  \MH_k(X)$ is an isomorphism.
        \end{enumerate}
    \end{proposition}
    \begin{proof}
        We show \ref{X_is_diago}$\Rightarrow $\ref{Z_eq_Z_cap_S_oplus_B}$\Rightarrow $\ref{incl_then_quot_is_isom}$\Rightarrow $\ref{X_is_diago}and \ref{X_is_diago}$\Leftrightarrow $\ref{decomp_sat}.
        \begin{itemize}
            \item[\ref{X_is_diago}$\Rightarrow $\ref{Z_eq_Z_cap_S_oplus_B}]
                By~\Cref{proposition:S_cap_B_is_trivial}, $\MZ_k(X)\cap S_k(X)$ and $\MB_k(X)$ intersect trivially.
                By diagonality, given $\sigma \in \MZ_k(X)$, there exists $\sigma _s \in  \MZ_k(X)\cap S_k(X)$ such that $\sigma $ and $\sigma _s$ are equal in homology, i.e. $\sigma -\sigma _s \in  \MB_k(X)$.
                Thus, $\MZ_k(X)\cap S_k(X) + \MB_k(X) = \MZ_k(X)$.
            \item[\ref{Z_eq_Z_cap_S_oplus_B}$\Rightarrow $\ref{incl_then_quot_is_isom}]
                Clear.
            \item[\ref{incl_then_quot_is_isom}$\Rightarrow $\ref{X_is_diago}]
                Clear.
            \item[\ref{X_is_diago}$\Rightarrow $\ref{decomp_sat}]
                Fix $\sigma  = \sigma _S + \sigma ' \in  \MZ_k(X)$ as in~\ref{decomp_sat}.
                By diagonality, there exists $\sigma _S' \in  \MZ_k(X)\cap S_k(X)$ and $\sigma '' \in  \MB_k(X)$ such that 
                $
                    \sigma _S + \sigma ' = \sigma _S' + \sigma '',
                $
                hence 
                $
                    \sigma _S - \sigma _S' = \sigma '' - \sigma '
                $.
                Since $\sigma _S - \sigma _S'$ has support saturated chains and $\sigma ''$ has no saturated chains in its support, necessarily $\sigma ' = \sigma ''$ and $\sigma _S = \sigma _S'$.
            \item[\ref{decomp_sat}$\Rightarrow $\ref{X_is_diago}]
                Clear.
        \end{itemize}

    \end{proof}

    In particular, it follows directly from the last item that:
    \begin{corollary}\label{corollary:diagonal_implies_no_torsion}
        A diagonal space has torsion-free homology.
    \end{corollary}

    To support our choice of terminology:
    \begin{proposition}
        A graph is diagonal in the sense of Hepworth and Willerton iff it is diagonal in the above sense. 
    \end{proposition}
    \begin{proof}
        In a graph, if $k\neq l$, no $k$-\ppath{} of length $l$ can be saturated.
        Thus, having vanishing homology outside the diagonal and having homology groups generated by (sums of) saturated \ppath{}s is equivalent.
    \end{proof}

    \begin{proposition}
        A diagonal Menger convex space has vanishing homology (except possibly at $k=0$).
    \end{proposition}
    \begin{proof}
        A Menger convex space has no saturated \ppath{}.
    \end{proof}

    In the next few propositions, we verify stability of diagonality under some usual constructions.

    \begin{proposition}\label{proposition:diagonality_preserved_by_filtrations}
        If $(U_\alpha )_\alpha $ is a filtration of a space $X$ such that each $U_\alpha $ is diagonal, then so is $X$. 
    \end{proposition}
    \begin{proof}
        Fix $\sigma  \in  \MZ_k(X)$, and write $\sigma  = \sigma _S + \sigma '$ where $\sigma _S$ consists of the saturated \ppath{}s of $\sigma $.
        Fix $\alpha $ large enough that it contains the support of $\sigma $, and for each non-saturated \ppath{} in $\sigma '$, also contains a “witness” to non-saturation of the \ppath{}.
        Then, $\sigma  \in  \MZ_k(X)$ and $\sigma _S$ still consists of saturated \ppath{}s \emph{in $U_\alpha $}, and $\sigma '$ of non-saturated \ppath{}s \emph{in $U_\alpha $}.
        Since $U_\alpha $ is diagonal, $\sigma ' \in  \MB_k(U_\alpha ) \leq  \MB_k(X)$.
    \end{proof}

    \begin{proposition}
        An $l^{1}$ product of diagonal spaces is diagonal.
    \end{proposition}
    \begin{proof}
        By applying the (metric) Künneth formula.
        Let $X,Y$ be diagonal spaces.
        For any fixed $n,l$, we have a short exact sequence
        \begin{alignat*}{10}
            0 \rightarrow  &\bigoplus _{n+l=k} \MH_{n}(X)\otimes \MH_{l}(Y) &\ \ovs{ \square }{\longrightarrow }\ &\MH_k(X\times Y) \rightarrow  \mathrm{Tor}(\dots ,\dots ) \rightarrow  0.\\
            \intertext{The torsion part being zero (\Cref{corollary:diagonal_implies_no_torsion}), an isomorphism}
                &\bigoplus _{n+l=k} \MH_{n}(X)\otimes \MH_{l}(Y) &\ \ovs{ \square }{\longrightarrow }\ &\MH_k(X\times Y)
        \end{alignat*}
        remains.
        Since each groups $\MH_{n}(X),\MH_{l}(Y)$ is generated by (sums of) saturated \ppath{}s, the whole LHS is generated by tensors of such, and $\MH_k(X\times Y)$ by their image.
        Noting that if $\zz$ is an interleaving of two \ppath{}s $\xx,\yy$ as in the definition of the map $ \square $, then $\zz$ is saturated iff both $\xx,\yy$ are, we conclude that $\MH_k(X\times Y)$ is generated by (sums of) saturated \ppath{}s.
    \end{proof}

    \begin{proposition}
        If $(X,Y,Z)$ is a gated decomposition and $Y,Z$ are convex and diagonal, then so is $X$.
    \end{proposition}
    \begin{proof}
        By applying a fragment of the (metric) Mayer-Vietoris sequence:
        We have an epimorphism:
        \[
            \MH_*(Y)\oplus \MH_*(Z) \rightarrow  \MH_*(X),
        \]
        and the image of saturated chains in either $Y$ or $Z$ are still saturated in $X$ by convexity.
    \end{proof}

    If $X,Y$ are two metric spaces, and $f:X\rightarrow Y$ is injective map, we say that $f$ \emph{preserves betweenness} if $y \in  [x,z]$ implies $fy \in  [fx,fz]$; and \emph{reflects betweenness} if $fy \in  [fx,fz]$ implies $y \in  [x,z]$.
    In case $f$ both preserves and reflects betweenness, we say it is a \emph{betweenness embedding}.
    If it is also surjective, it becomes a \emph{betweenness isomorphism}.

    \begin{proposition}
        If $f:X\rightarrow Y$ is a betweenness embedding, then $f$ induces a morphism of chain complexes:
        \begin{alignat*}{10}
            f_* :\MC_*(X) &\rightarrow  \MC_*(Y)\\
            (x_0,\dots ,x_k) &\mapsto  (fx_0,\dots ,fx_k).
        \end{alignat*}
        If $f$ is bijective, this turns into an isomorphism.
    \end{proposition}
    \begin{proof}
        Injectivity plus the betweenness preserving+reflecting makes the \ppath{} map commute with boundaries.
    \end{proof}

    \begin{proposition}\label{proposition:diagonality_preserved_by_betweenness}
        If $f:X\rightarrow Y$ is a betweenness isomorphism between metric spaces, and $Y$ is diagonal, then so is $X$, and vice versa.
    \end{proposition}
    \begin{proof}
        Since $f$ preserves and reflects betweenness, images of saturated \ppath{}s are saturated, and vice versa.
        The same can be said of homological cycles and boundaries.
    \end{proof}

    \begin{proposition}\label{proposition:diagonality_preserved_by_retracts}
        Retracts of diagonal spaces are diagonal.
    \end{proposition}
    \begin{proof}
        Let $f:X\rightarrow Y$ a retraction.
        
        If $\yy$ is a \ppath{} in $Y$, then $\yy$ is saturated in $Y$ iff it is saturated in $X$.
        Indeed, if, say, there exists some $x\in X$ strictly between $y_i$ and $y_{i+1}$, then, since $f$ is non-expanding and fixes $y_i,y_{i+1}$, it follows that $fx$ is strictly between $y_i$ and $y_{i+1}$.
        This shows that non-saturatedness in $X$ implies the same in $Y$.
        Conversely, if $\yy$ is saturated in $X$ it is a fortiori also in $Y$.

        Now, fix $\sigma  = \sigma _S + \sigma ' \in  \MZ_k(Y)$ a cycle.
        Since $Y\subseteq X$, $\sigma $ is still a cycle in $X$, and its decomposition into “saturated+non-saturated” in $X$ is still $\sigma  = \sigma _S + \sigma '$.
        Thus, assuming $X$ is diagonal, $\sigma '\in  \MB_k(X)$; that is, there exists some $\tau \in \MC_{k+1}(X)$ with $\sigma ' = \partial \tau $.
        But then, applying $f_*:\MC_*(X) \rightarrow  \MC_*(Y)$, we get $f_*\sigma ' = f_*\partial \tau  = \partial f_*\tau $, and since $\sigma '$ has support in $Y$, $f_*\sigma ' = \sigma '$.
        Thus $\sigma ' = \partial f_*\tau  \in  \MB_k(Y)$. 
        This shows that $Y$ is diagonal.
    \end{proof}

    \section{Median Spaces are Diagonal}

    We will need the following fact due to Avann:
    \begin{proposition}[\cite{A}]\label{proposition:finite_median_spaces_are_almost_graphs}
        If $X$ is a finite median space, there exists a finite graph $G(X)$ and a betweenness isomorphism $\phi :X\rightarrow G(X)$.
    \end{proposition}
    
    Another important property of median spaces:
    \begin{proposition}
        Any median space has a filtration by finite median subspaces.
    \end{proposition}
    \begin{proof}
        It can be seen that any finite subset of a median space has a finite so-called \emph{median hull}, that is, a smallest median space containing the set in question (see~\cite[p.7]{BC}).
        Taking all such finite median hulls yields a desired filtration.
    \end{proof}

    It is now easy to conclude that
    \begin{proposition}
        Median spaces are diagonal.
    \end{proposition}
    \begin{proof}
        Finite median spaces are diagonal since so are (finite) median graphs, and by applying~\Cref{proposition:finite_median_spaces_are_almost_graphs,proposition:diagonality_preserved_by_betweenness}. 
    \end{proof}

    \begin{corollary}
        Median Menger convex spaces have vanishing homology (except possibly at $k=0$).
    \end{corollary}

    \section{Betweenness}\label{section:betweenness}

    In this section, the notion of magnitude for a kind of “axiomatic betweenness spaces” is developed.
    The first step is to define a  notion of “betweenness space”, which we do by isolating the key properties of betweenness in metric spaces needed in the construction of the magnitude complex.
    It will then be easy to see that the magnitude complex of a metric space is a special case of the magnitude complex of its underlying betweenness space.
    The harder (and less satisfying) part is functoriality: a naïve definition of morphisms for betweenness spaces will not make $1$-Lipschitz maps a special case of these naïve morphisms (in the same sense that metric spaces can be seen as special cases of betweenness spaces), and hence, the resulting category will be of no help as a generalization of metric spaces in the context of magnitude homology.
    The proposed solution is to relax the definition of morphisms, thus getting a new intermediary category through which the original magnitude complex functor will now indeed factor.

    \subsection{Magnitude of Metric Spaces, Recast}

    Let $\Met$ denote the category of metric spaces and $1$-Lipschitz maps, $\ptsSet$ of pointed simplicial sets, $\ChZ$ of chain complexes over $\ZZ $, and $(\ChZ\leq \ChZ)$ of “chain complexes and subcomplexes” over $\ZZ $.
    The magnitude complex functor $\MC_*$ can be decomposed into the composite:
    \[\begin{tikzcd}
        \Met \ar[r, "\MS"] & \ptsSet \ar[r] & (\ChZ \leq  \ChZ) \ar[r,"\text{quotient}"] & \ChZ
    \end{tikzcd}\]
    where the second arrow sends a pointed simplicial set to the usual chain complex freely spanned by simplices and its subcomplex spanned by basepoints and degenerate simplices, and the last arrow takes the quotient of such a pair.
    Strictly speaking, $\MS$ has not yet been defined on $1$-Lipschitz maps:
    If $f:X\rightarrow Y$ is $1$-Lipschitz, then
    \begin{alignat*}{10}
        \MS(f) : \MS(X) &\rightarrow  \MS(Y)\\
                 \xx &\mapsto  \begin{cases}
                 f\xx & \text{ if  $l(f\xx) = l(\xx)$;} \\
                    \pt  & \text {otherwise.}
                 \end{cases}
    \end{alignat*}
    It is easily checked that this description is equivalent to the usual definition of the magnitude chain complex functor.

    From now on, we will forget about everything but the first arrow $\MS$, since this is where the bulk of “magnitude” happens.

    \subsection{Betweenness Space}

    We can now define betweenness spaces.
    Different notions of betweenness spaces have already been studied at length (see e.g.~\cite{Bowditch,Coppel}) but no comparison will be made here, since our purpose is specific to magnitude.

    \begin{definition}[Betweenness space]
        A \emph{betweenness space} is a set $X$  endowed with a map $[\bullet ,\bullet ]:X\times X \rightarrow  \mathcal{P}(X)$ satisfying the following axioms:
        \begin{enumerate}[label=B\arabic*,ref=B\arabic*]
            \item\label{axiom:endpoints_in_interval}
                $x,z \in  [x,z]$;
            \item\label{axiom:subintervals}
                $y \in  [x,z]$ implies $[x,y] \subseteq  [x,z]$ and $[y,z] \subseteq [x,z]$;
            \item\label{axiom:directionality}
                $w,y \in  [x,z]$ implies: $w \in  [x,y]$ iff $y \in  [w,z]$.
        \end{enumerate}
    \end{definition}

    Naturally, we expect metric spaces with their interval structure to be betweenness spaces.
    \begin{example}[Metric spaces]
        For a metric space $X$, and $x,z\in X$, we let $[x,z]$ be the set of points between $x$ and $z$ (realising the triangle inequality).
        It is easy to see that~\ref{axiom:endpoints_in_interval} is verified.
        For \ref{axiom:subintervals}, fix $y\in [x,z]$ and $w\in [x,y]$, so that 
        \begin{alignat}{10}
            |xy| + |yz| &= |xz|, \label{sub_xyz}\\
            |xw|+|wy|  &=|xy|. \label{sub_xwy}
        \end{alignat}
        Then:
        \[
            |xw| + |wz| \leq  |xw| + |wy| + |yz| \ovs{\cref{sub_xwy}}{=} |xy|+ |yz| \ovs{\cref{sub_xyz}}= |xz|,
        \]
        hence $w\in [x,z]$, and ~\ref{axiom:subintervals} holds.
        For ~\ref{axiom:directionality}, fix $w,y\in [x,z]$.
        By symmetry, it is enough to check that $w\in [x,y]$ implies $y\in [w,z]$.
        From
        \begin{alignat}{10}
            |xw| + |wz| = |xz|, \label{dir_xwz}\\
            |xy| + |yz| = |xz|, \label{dir_xyz}\\
            |xw| + |wy| = |xy|, \label{dir_xwy}
        \end{alignat}
        we get
        \begin{alignat*}{10}
            |wy| + |yz| \ovs{\text{\cref{dir_xwy} + \cref{dir_xyz}}}{=} |xy| - |xw| + |xz| - |xy| = |xz|-|xw| \ovs{\text{\cref{dir_xwz}}}{=} |wz|,
        \end{alignat*}
        which shows $y\in [wz]$.

    \end{example}
    
    Let us write $UX$ for the betweenness space associated to the metric space $X$.

    \begin{example}[“Complete spaces”]
        Given any set $X$, one can set $[x,y] := X$, which obviously satisfies the axioms. 
    \end{example}

    \begin{example}[“Incomplete spaces”]
        Inversely to the “complete space” on $X$, one can set $[x,z] := \{x,z\}$, and the axioms are verified by a simple case check.
    \end{example}
  	
  	\begin{example}[Posets]
  		For a poset $(\mathcal{P},\le)$, and $x,z\in \mathcal{P}$, we let $[x,z]$ be the set $\{ y \;\vert\; x\le y \text{ and } y\le z\} \cup \{ x,z \}$. The first axiom of betweenness follows from the definition of the intervals. The other two are a consequence of transitivity of $\le$.
  	\end{example}
  
    As explained above, to define the magnitude homology of a betweenness space, we focus only on the first step, which is defining a corresponding pointed simplicial set:

    \begin{definition}[Magnitude simplicial set]\label{definition:bet_magnitude_sset}
        If $X$ is a betweenness space, we define the pointed simplicial set $\MS(X)$ as having $k$-simplices the $(k+1)$-tuples of points $\langle x_0,\dots ,x_{k}\rangle :[k]\rightarrow X$ in $X$, plus basepoint simplices $\pt_n$, along with face and degeneracy maps defined by:
        \begin{alignat*}{10}
            \ddd_{k,i} \langle x_0,\dots ,x_i,\dots ,x_k\rangle  &:= \begin{cases}
                \langle x_0,\dots ,\widehat{x_i},\dots ,x_k\rangle  &\text{if $x_i \in  [x_{i-1},x_{i+1}]$ }\\
                \pt_{k-1}                 &\text{otherwise,}
            \end{cases}
            \intertext{and}
                \sss_{k,i} \langle x_0,\dots ,x_i,\dots ,x_k\rangle  &:= \langle x_0,\dots ,x_i,x_i,\dots ,x_k\rangle ,\\
            \intertext{and on basepoints:}
                \ddd_{k,i} \pt_k &:= \pt_{k-1}\\
                \sss_{k,i} \pt_k &:= \pt_{k+1}.
        \end{alignat*}
    \end{definition}
   
    The simplicial identities are all easily verified (using~\ref{axiom:endpoints_in_interval} for the ones involving both $\ddd$ and $\sss$), with the exception of the one involving $\ddd$ only:

    \begin{lemma}
        For all $i<j<k$, $\ddd_{k-1,i}\ddd_{k,j} = \ddd_{k-1,j-1}\ddd_{k,i}$.
    \end{lemma}
    \begin{proof}
        If $j>i+1$, the indices are “far enough” that the different $\ddd_{\bullet ,\bullet }$s commute; the case $j=i+1$ remains.
        For simplicity, let us consider a $3$-\ppath{} $\pp =(x,y,w,z)$ with $i=1,j=2$ (the general case is similar).
        If $\ddd_1\ddd_2\pp \neq  \ddd_2\ddd_1\pp$, necessarily one side is $\pt$, and not the other.
        By symmetry, we may assume $\ddd_1\ddd_2\pp=\pt$ and $\ddd_1\ddd_1\pp\neq \pt$.
        This means:
        \begin{enumerate}
            \item 
                $w\in [y,z]$ and $y\in [x,z]$ (the $\ddd_1\ddd_1\pp\neq \pt$ part).
            \item 
                Either $y\notin [x,w]$, or $y\in [x,w] \wedge  w \notin  [x,z]$ (the $\ddd_1\ddd_2\pp=\pt$ part).
        \end{enumerate}
        We show that both are impossible:

        From the first point and~\ref{axiom:subintervals}, $w \in  [x,z]$.
        From ~\ref{axiom:directionality}, $w\in [y,z]$ implies $y\in [x,w]$, so that the first part of the second point is impossible, i.e. $y\in [x,w]$.
        But then, the second part is also known to be impossible, since $w\in [x,z]$.
    \end{proof}

    \begin{example}[Metric spaces]\label{example:MS_X_eq_MS_UX}
        Letting $UX$ denote the betweenness space associated to a metric space $X$, it is direct to see that $\MS(UX) = \MS(X)$.
    \end{example}

    \subsection{Morphisms and Functoriality}

    As we know, $\MS$ is a functor from $\Met$ to $\ptsSet$.
    We have also seen that any metric space defines in particular a betweenness space, through the “sort of forgetful” map $U$, and, as seen in~\Cref{example:MS_X_eq_MS_UX}, $\MS(X)$ factors through $U$.
    We would like to make this factoring into one of functors, for which we first need to define morphisms on betweenness spaces.

    An (unsatisfying) first guess is the following:
    \begin{definition}[Betweenness morphisms]
        A \emph{betweenness morphism} is a map $f:X\rightarrow Y$ between betweenness spaces $X$ and $Y$ that:
        \begin{itemize}
            \item 
                preserves betweenness ($x\in [x',x'']$ implies $fx\in [fx',fx'']$); and
            \item
                reflects betweenness ($fx\in [fx',fx'']$ implies $x\in [x',x'']$).
        \end{itemize}

        The category $\Bet$ has objects betweenness spaces and morphisms betweenness morphisms.
    \end{definition}
    It is easily verified that $\Bet$ indeed forms a category, and for a betweenness morphisms $f:X\rightarrow Y$, letting\begin{alignat*}{10}
        \MS(f): \MS(X) &\rightarrow  \MS(Y)\\
        \langle x_1,\dots ,x_n\rangle  &\mapsto  \langle fx_1,\dots , fx_n\rangle \\
        \pt &\mapsto  \pt.
    \end{alignat*}
    will quite clearly yield a functor $\MS:\Bet \rightarrow  \ptsSet$.
  
    But there is a catch: $1$-Lipschitz maps are not required to either preserve or reflect betweenness.
    Hence, the “sort of forgetful” map $U$ does not work on morphisms.
    
    Thus, we must find an appropriately looser category, $\Bet_-$, for which we do get a functor.
    The picture we will reach is the following:

    
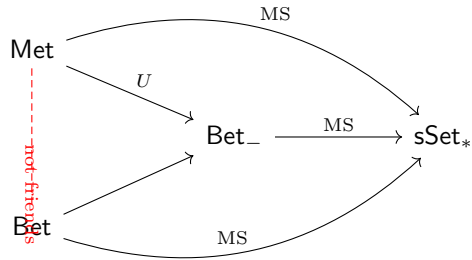
\[\begin{tikzcd}
        \Met \ar[rrd,"U"] \arrow[sloped,above,dash,red,dashed]{dd}{\text{not friends}} \ar[bend left, rrrrd, "\MS"] & &        & &    \\
                                                                                                                        & & \Bet_- \ar[rr, "\MS"] & & \ptsSet  \\
        \Bet \ar[rru] \ar[bend right, rrrru, "\MS"] & &        & & 
    \end{tikzcd}\]

    In a sense, $\Bet_-$ will serve as an intermediary category — one foot in betweenness spaces, the other in pointed simplicial sets — whose morphisms will encompass the good behaviour of $1$-Lipschitz maps and betweenness morphisms.

    \begin{definition}[Good \ppath{}s]
        Let $f:X\rightarrow Y$ be a map between betweenness spaces $X$ and $Y$.
        An \emph{$f$-good} set of \ppath{}s in $X$ is a set $G$ of (non basepoint) simplices of $\MS(X)$ satisfying:
        \begin{enumerate}[label=G\arabic*,ref=G\arabic*]
            \item\label{axiom:good}
                If $\xx = \langle x_0,\dots ,x_k\rangle \in G$ then: $fx_i \in  [fx_{i-1},fx_{i+1}]$ iff $x_i \in [x_{i-1},x_{i+1}]$ and $\ddd_i x\in G$.
            \item\label{axiom:bad}
                If $\xx =\langle x_0,\dots ,x_k\rangle  \notin  G$  and $x_i \in [x_{i-1},x_{i+1}]$, then $\ddd_i x \notin  G$.
        \end{enumerate}
        We unsurprisingly call a \ppath{} \emph{good} if it lies in $G$ (given by the context), and bad otherwise.
    \end{definition}
    If we view basepoints as bad paths,~\ref{axiom:bad} simply states that $\MS(X)-G$ is closed under $\ddd$.

    We can now define our category: 
    \begin{definition}[$\Bet_-$]
        The category $\Bet_-$ has as objects betweenness spaces, and as morphisms from $X$ to $Y$ pairs $(f:X\rightarrow Y,G_f)$ where $f$ is a map (of sets) from $X$ to $Y$ and $G_f$ is an $f$-good set of paths.
        The identity morphism on $X$ has $G_{\Id} = \MS(X)-\{\pt_n\}$, and if $(f:X\rightarrow Y,G_f), (g:Y\rightarrow Z,G_g)$ are two morphisms, their composite is $(g\circ f,G_f\cap f^{-1}[G_g])$.
    \end{definition}

    This definition is rather ad-hoc; its purpose is to capture the functoriality of $\MS$ originally restricted to $1$-Lipschitz maps, but in a now broader context.
    
    It is not obvious a priori that the composite of two morphisms is again a morphism, that is, that $G_f\cap f^{-1}[G_g]$ satisfies the requirements for a $g\circ f$-good set.

    \begin{lemma}[Good sets composition]
        If $(f:X\rightarrow Y,G_f), (g:Y\rightarrow Z,G_g)$ are two morphisms, then $G_f \cap  f^{-1}[G_g]$ is a $g\circ f$-good set.
    \end{lemma}
    \begin{proof}
        Fix $\xx$ a \ppath{} and let $\yy := f(\xx)$.

        Assume first that $\xx \in  G_f \cap  f^{-1}[G_g]$, i.e. $\xx \in  G_f, \yy \in  G_g$.
        Then:
        \begin{alignat*}{10}
                            & gf\xx_i \in  [gf\xx_{i-1},gf\xx_{i+1}] &\\
            \quad \Leftrightarrow  \quad   & g\yy_i \in  [g\yy_{i-1},g\yy_{i+1}]    &\qquad \text{by definition of $\yy$,}\\
            \quad \Leftrightarrow  \quad   & \yy_i \in  [\yy_{i-1},\yy_{i+1}] \text{ and } \ddd_i \yy \in  G_g   &\qquad \text{ since $\yy \in  G_g$ and by~\ref{axiom:good} for $G_g$,}\\
            \quad \Leftrightarrow  \quad   & f\xx_i \in  [f\xx_{i-1},f\xx_{i+1}] \text{ and } \ddd_i \yy \in  G_g   &\qquad \text{ by definition of $\yy$,}\\
            \quad \Leftrightarrow  \quad   & \xx_i \in  [\xx_{i-1},\xx_{i+1}] \text{ and } \ddd_i\xx \in  G_f \text{ and } \ddd_i \yy \in  G_g   &\qquad \text{ since $\xx \in  G_f$ and by~\ref{axiom:good} for $G_f$,}\\
            \quad \Leftrightarrow  \quad   & \xx_i \in  [\xx_{i-1},\xx_{i+1}] \text{ and } \ddd_i\xx \in  G_f\cap f^{-1}[G_g]   &
        \end{alignat*}
        which verifies~\ref{axiom:good} for $G_f \cap  f^{-1}[G_g]$.

        Assume now that $\xx \notin  G_f \cap  f^{-1}[G_g]$ and that $x_i \in  [x_{i-1},x_{i+1}]$.

        If $\xx \notin  G_f$, then $\ddd_i \xx \notin  G_f$ (by~\ref{axiom:bad} for $G_f$), and hence $\ddd_i\xx\notin  G_f\cap f^{-1}[G_g]$.
        If $\xx \in  G_f$ but $\yy \notin  G_g$, then (since $\xx\in G_f$ and by~\ref{axiom:good} for $G_f$) $y_i \in  [y_{i-1},y_{i+1}]$
        Therefore (since $\yy \notin  G_g$ and by~\ref{axiom:bad} for $G_g$) $\ddd_i \yy \notin  G_g$, hence $\ddd_i \xx \notin  G_f\cap f^{-1}[G_g]$, which verifies ~\ref{axiom:bad} for $G_f \cap  f^{-1}[G_g]$.
        
    \end{proof}
    
    Thus, $\Bet_-$ really is a category, and we can associate to each of its morphism a simplicial map: 

    \begin{definition}[$\MS$ on morphisms]
        If $(f:X\rightarrow Y,G_f)$ is a morphism in $\Bet_-$, we define:
        \begin{alignat*}{10}
            \MS(f): \MS(X) &\rightarrow  \MS(Y)\\
            \xx &\mapsto  \begin{cases}
                f(x) & \text{ if $\xx \in  G_f$}\\
                \pt    & \text{ otherwise.}
            \end{cases}
        \end{alignat*}
    \end{definition}
    
    \begin{lemma}[$\MS(f)$ is a pointed simplicial map]
        If $(f:X\rightarrow Y,G_f)$ is a morphism in $\Bet_-$, then $\MS(f)$ is a morphism in $\ptsSet$, i.e. preserves basepoints and commutes with $\ddd$ and $\sss$. 
    \end{lemma}
    \begin{proof}
        That $\MS(f)$ sends basepoints to basepoints and commute with $\sss$ is clear.
        Let us check that $\ddd_i \circ  \MS(f) = \MS(f) \circ  \ddd_i$.
        For simplicity, consider a simplex $\xx=\langle x_0,x_1,x_2\rangle $ with $i=1$, and let $y_i:=f(x_i)$.
        Checking equality reduces to checking that one side is equal to $\pt$ iff the other is.
        \begin{itemize}
            \item
                If $\xx$ is a bad path, then $\MS(f)(\xx) = \pt$, so $\ddd_1\circ \MS(f)\xx = \pt$.
                If also $x_1\in [x_0,x_2]$, then $\ddd_1 \xx = \langle x_0,x_2\rangle $ is again a bad path by~\ref{axiom:bad}, so that $\MS(f) \circ  \ddd_1 \xx = \pt$.
                Otherwise, $\ddd_1 \xx = \pt$ and, again, $\MS(f) \circ  \ddd_1 \xx =\pt$.
            \item
                If $\xx$ is a good path, then $\MS(f)\xx = \langle y_0,y_1,y_2\rangle  := \yy \neq  \pt$.
                If $y_1\in [y_0,y_2]$, then $x_1\in [x_0,x_2]$ and $\ddd_1\xx$ is a good path by~\ref{axiom:good}.
                Then, $\MS(f)\circ \ddd_1 \xx = \langle y_0,y_2\rangle  = \ddd_1\circ \MS(f) \xx$.
                Otherwise, either $x_1 \notin  [x_0,x_2]$ or $\ddd_1\xx$ is bad by~\ref{axiom:good}, both of which imply that $\MS(f)\circ \ddd_1 \xx = \pt = \ddd_1\circ \MS(f) \xx$.
        \end{itemize}

    \end{proof}

    Let us now check functoriality:
    \begin{lemma}
        $\MS(g\circ f) = \MS(g) \circ  \MS(f)$ and $\MS(\Id) =\Id$.
    \end{lemma}
    \begin{proof}
        Since the identity morphism has all non basepoints simplices as good, it is clear that $\MS(\Id) =\Id$.
        For compositions, fix $\xx$ and let $\yy := f(\xx)$.
        Clearly, it is enough to show that $\MS(g\circ f)(\xx) = \pt$ iff $ \MS(g) \circ  \MS(f)(\xx) = \pt$.
        The former is $\pt$ iff $\xx \notin  G_f \cap  f^{-1}[G_g]$, which is iff $\xx \notin  G_f$ or $\xx \in  G_f$ but $\yy \notin  G_g$, which is exactly the condition for the latter being $\pt$.
    \end{proof}

    We now have three different magnitude functors $\MS$ at hand: one for each of $\Met,\Bet,\Bet_-$.
    It is easy to see that the first two factor through the third on objects, and being able to \emph{choose} good sets is what allows the factoring on morphisms: 

    \begin{example}[Metric spaces]\label{example:good_set_for_Lipschitz}
        If metric spaces $X,Y$ are viewed as betweenness spaces, a $1$-Lipschitz map $f:X\rightarrow Y$ may be endowed with the following set of good \ppath{}s:
        \[
            G_f := \{\xx=(x_0,\dots ,x_k)\ :\ l(f\xx) = l(\xx)\},
        \]
        where $l(\bullet )$ denotes the length of a \ppath{}, as usual.
        Furthermore, if $f:X\rightarrow Y$ and $g:Y\rightarrow Z$ are $1$-Lipschitz, then:
        \[
            G_{g\circ f} = G_{f} \cap  f^{-1}[G_g].
        \]
        Indeed, $g\circ f(\xx)$ has the same length as $\xx$ iff both $f(\xx)$ and $g(f(\xx))$ do, from which the equality is obvious.
    
        We see now that the “sort of forgetful” map $U$ \emph{does} define a functor to $\Bet_-$ when appropriately defined on morphisms.
        Furthermore, the definitions imply that $\MS:\Met \rightarrow  \ptsSet$ now factors through $U$, as hoped.
    \end{example}

    \begin{example}[Betweenness morphisms]
        If $f:X\rightarrow Y$ is a betweenness morphism, then taking $G_f = \MS(X) - \{\pt\}$ is a valid choice of good paths, and, again, we conclude that $\MS:\Bet \rightarrow  \ptsSet$ factors through $\Bet_-$.
    \end{example}

    The following lemma 

    \begin{lemma}
        For any map $f:X\rightarrow Y$ between betweenness spaces, there exists a maximal good set of \ppath{}s.
    \end{lemma}
    \begin{proof}
        Let $\mathcal{G}$ be the set of all good sets of \ppath{}s for $f$.
        We show that $\bigcup \mathcal{G}$ is also a good set of \ppath{}s.
        \begin{enumerate}
            \item
                Assume $\xx \in  \bigcup \mathcal{G}$.

                Assume first $fx_i \in  [fx_{i-1},fx_{i+1}]$ and let  $G\in \mathcal{G}$ with $\xx\in  G$.
                Then $x_i \in  [x_{i-1},x_{i+1}]$ and $\ddd_i \xx \in  G$ (since $G$ is a good set), and we conclude $\xx \in  \bigcup \mathcal{G}$.
                
                Conversely, if $x_i \in  [x_{i-1},x_{i+1}]$ and $\ddd_i \bigcup \mathcal{G}$, then $\ddd_i\xx \in  G$ for some good set $G$.
                If $\xx \notin  G$, we would have $\ddd_i \xx \notin  G$, (second rule of good sets).
                We conclude that $\xx \in  G$, hence $\xx \in  \bigcup \mathcal{G}$.
            \item 
                If $\xx \notin  \bigcup \mathcal{G}$, then $\xx$ lies in no good set.
                Assuming $x_i \in  [x_{i-1},x_{i+1}]$, we get, for each good set $G$, that $\ddd_i x \notin  G$.
                Thus, $\ddd_i x \notin  \bigcup \mathcal{G}$. 
        \end{enumerate}
    \end{proof}
    \begin{definition}[Maximal good set]
        If $f:X\rightarrow Y$, write $\hG(f)$ for the maximal good set of \ppath{}s of $f$.
    \end{definition}

    \subsection{Künneth and Mayer-Vietoris Again}
    
    We can now verify that both Künneth and Mayer-Vietoris still make sense for betweenness spaces.

    If $X,Y$ are betweenness spaces, their product is the betweenness space with underlying set the Cartesian product $X\times Y$ and interval structure:
    \[
        [(x_1,y_1),(x_2,y_2)] := [x_1,x_2] \times  [y_1,y_2].
    \]
    A routine verification shows that the axioms (\ref{axiom:endpoints_in_interval},\ref{axiom:subintervals},\ref{axiom:directionality}) are satisfied.

    Note that if $(f:X\rightarrow X',G_f),(g:Y\rightarrow Y',G_g)$ are morphisms, then the set
    \[
        G_{f\times g} := \{\xx\times \yy := \langle (x_0,y_0),\dots ,(x_k,y_k)\rangle \ |\ k\geq 0,\ \xx = \langle x_0,\dots ,x_k\rangle  \in  G_f,\ \yy = \langle y_0,\dots ,y_k\rangle \in G_g\}
    \]
    is a good set for $f\times g:X\times Y\rightarrow X'\times Y'$, so that $(f\times g:X\times Y\rightarrow X'\times Y',G_{f\times g})$ is also a morphism.

    \begin{proposition}[Künneth theorem – betweenness setting]\label{proposition:bet_Kunneth}
        Let $X,Y$ be betweenness spaces, and let $X\times Y$ be their product.
        The cross product map
        \begin{alignat*}{10}
            \square : \MC_*(X) \otimes  \MC_*(Y) &\rightarrow  \MC_*(X\times Y)\\
            \xx\otimes \yy \ &\ovs{\square }{\mapsto } \sum _{\sigma \in \st{n,l}} \sgn(\sigma ) (\xx \ovs{\sigma }{\times } \yy)\\
            \intertext{induces a morphism}
            \MH_*(X)\otimes \MH_*(Y) &\ \ovs{ \square }{\rightarrow }\ \MH_*(X\times Y)\\
                [f]\otimes [g]         &\mapsto  [f \square  g]
        \end{alignat*}
        which fits into a (natural on $\Bet_-$) short exact sequence:
        \[
            0 \rightarrow  \MH_*(X)\otimes \MH_*(Y)\ \ovs{ \square }{\rightarrow }\ \MH_*(X\times Y) \rightarrow  \mathrm{Tor}(\MH_{*-1}(X),\MH_*(Y)) \rightarrow  0.
        \]
    \end{proposition}
    \begin{proof}
        As the metric case,
        start with betweenness spaces $X,Y$. 
        The map
        \[
            \square: \MS(X) \wedge  \MS(Y) \rightarrow  \MS(X\times Y); \quad [\xx,\yy] \mapsto  \langle \xx,\yy\rangle 
        \]
        is a natural morphism of pointed simplicial sets, in the sense that if $X',Y'$ and morphisms $(f:X\rightarrow X',G_f),\ (g:Y\rightarrow Y',G_g)$ are given, we get a commuting square
        \[\begin{tikzcd}
            \MS(X) \wedge  \MS(Y) \dar["\MS(f)\wedge \MS(g)"]\rar["\square"] & \MS(X\times Y) \dar["\MS(f\times g)"] \\
            \MS(X') \wedge  \MS(Y) \rar["\square"] & \MS(X'\times Y').
        \end{tikzcd}\]
        
        Applying $\NR_*$ yields $\MC_*(X),\MC_*(Y)$ and $\MC_*(X\times Y)$ and $\REZ$ the desired isomorphism.

    \end{proof}

    In the case of Mayer-Vietoris, we first need to translate the notion of “gated decomposition”.

    Let $X$ be a betweenness space, $Y,Z$ subsets of $X$, such that $X=Y\cup Z$, and $\pi :Z\rightarrow Y\cap Z$ a map satisfying:
    \begin{itemize}
        \item
            For all $z\in Z,y\in Y$: $\pi (z) \in  [z,y]$;
        \item
            $\pi |_{Y\cap Z} = \Id_{Y\cap Z}$;
        \item
            $\hG(\pi )$ contains $\MS(Y\cap Z) - \{\pt_n\}$.
    \end{itemize}
    We again call this data a \emph{gated decomposition} of $X$.
    The (last two) requirements for $\pi $ imply in particular that it (along with its maximal good set) defines a retraction in the category $\Bet_-$, of left inverse the inclusion $\iota $:
    More precisely $(\pi :Z\rightarrow Y\cap Z, \hG(\pi ))$ has left inverse $(\iota :Y\cap Z\rightarrow Z,\MS(Y\cap Z)-\{\pt\})$.

    \begin{proposition}[Excision – betweenness setting]\label{proposition:bet_excision}
        If $X,Y,Z,\pi $ is the data of a gated decomposition of $X$, then the inclusion
        \[
            \MC_*(Y,Z) \hookrightarrow  \MC_*(X)
        \]
        is a quasi-isomorphism.
    \end{proposition}
    \begin{proof}
        Similar as in the metric case.
        For the acyclicity of the complex $A_*(a,b)$ of~\Cref{proposition:A_is_acyclic}, the three properties of $\pi $ are used.
        All subsequent arguments build only upon \Cref{proposition:A_is_acyclic} and the (set-theoretic) decomposition $X=Y\cup Z$ (i.e. nothing more is used), hence they still work in the current context.
    \end{proof}

    \begin{theorem}[Mayer-Vietoris – betweenness setting]
        If $X,Y,Z,\pi $ is the data of a gated decomposition of $X$, then the obvious morphisms
        \[
            j_Y:W\rightarrow Y,\ j_Z:W\rightarrow Z,\ i_Y:Y\rightarrow X,\ i_Z:Z\rightarrow X
        \] 
        induce a short exact sequence:
        \[
            0 \rightarrow  \MH_*(W) \ovs{\langle (j_Y)_*, -(j_Z)_*\rangle}{\longrightarrow } \MH_*(Y)\oplus \MH_*(Z) \ovs{(i_Y)_*\oplus (i_Z)_*}{\longrightarrow } \MH_*(X) \rightarrow  0.
        \]
    \end{theorem}
    \begin{proof}
        Same as in the metric case.
        Note that all inclusions $\iota $ yield morphisms $(\iota ,\MS(\dom\iota )-\{\pt\})$ in $\Bet_-$, and as explained above, $\pi $ defines a retraction in $\Bet_-$.
    \end{proof}

    \bibliographystyle{unsrt} 
    \bibliography{biblio}
\end{document}